\documentclass[a4paper,11pt,oneside]{article}

\usepackage[T1]{fontenc}
\usepackage[utf8]{inputenc}
\usepackage{amsfonts}  
\usepackage{amsmath}
\usepackage{graphicx} 
\usepackage{enumerate}
\usepackage{amsthm}
\usepackage{caption}
\usepackage{marvosym}
\usepackage{xcolor}
\usepackage{amssymb}
\usepackage{amsbsy}

\newcommand{\eqdef}{\ensuremath{\stackrel{\mathrm{def}}{=}}}
\def\eqsp{\;}

\newcommand{\PE}{\mathbb{E}}
\newcommand{\PP}{\mathbb{P}}
\newcommand{\T}{\theta}
\newcommand{\X}{\mathbf{X}}

\newcommand{\anneau}{H}
\newcommand{\bornea}{\xi}
\newcommand{\tetaset}{\Theta}
\newcommand{\refmes}{\mu}
\newcommand{\mesobs}{\pi_u}
\newcommand{\vitborn}{\Gamma}
\newcommand{\vit}{\gamma}

\def\rmd{\mathrm{d}}
\def\Xsigma{\mathcal{X}}
\def\param{v}
\def\Rset{\mathbb{R}}
\def\Nset{\mathbb{N}}
\def\prior{\chi}
\def\un{\mathbf{1}}
\def\InterractionProba{\varepsilon}

\newcounter{hypA}
\newenvironment{hypA}{\refstepcounter{hypA}\begin{itemize}
  \item[{\bf R\arabic{hypA}}]}{\end{itemize}}

\newcounter{hypEE}
\newenvironment{hypEE}{\refstepcounter{hypEE}\begin{itemize}
  \item[{\bf E\arabic{hypEE}}]}{\end{itemize}}

\newtheorem{theorem}{Theorem}[section]
\newtheorem{proposition}[theorem]{Proposition}

\newtheorem{lemma}[theorem]{Lemma}

\title{Adaptive Equi-Energy Sampler: Convergence and Illustration}
\author{Amandine SCHRECK, Gersende FORT and Eric MOULINES}

\begin{document}
 
\maketitle

\section*{Abstract}

  Markov chain Monte Carlo (MCMC) methods allow to sample a distribution known
  up to a multiplicative constant. Classical MCMC samplers are known to have
  very poor mixing properties when sampling multimodal distributions. The
  Equi-Energy sampler is an interacting MCMC sampler proposed by Kou, Zhou and
  Wong in 2006 to sample difficult multimodal distributions. This algorithm
  runs several chains at different temperatures in parallel, and allow
  lower-tempered chains to jump to a state from a higher-tempered chain having
  an energy `close' to that of the current state.  A major drawback of this
  algorithm is that it depends on many design parameters and thus, requires a
  significant effort to tune these parameters.
  
  In this paper, we introduce an Adaptive Equi-Energy (AEE) sampler which
  automates the choice of the selection mecanism when jumping onto a state of
  the higher-temperature chain. We prove the ergodicity and a strong law of large
  numbers for AEE, and for the original Equi-Energy sampler as well.  Finally,
  we apply our algorithm to motif sampling in DNA sequences.

\vspace{4cm}

Keywords: interacting Markov chain Monte Carlo, adaptive sampler, equi-energy sampler,  ergodicity, law of large numbers, motif sampling.

\vspace{2cm}

Author's email addresses: first-name.last-name@telecom-paristech.fr

\medskip

This work is partially supported by the French National Research Agency, under the program ANR-08-BLAN-0218 BigMC.

\clearpage


\section{Introduction}
Markov Chain Monte Carlo (MCMC) methods are well-known tools for sampling  a target distribution $\pi$ known up to a multiplicative constant.
MCMC algorithms sample $\pi$ by constructing a Markov chain admitting $\pi$ as unique
invariant distribution.
A canonical example is the the Metropolis-Hastings algorithm \cite{metropolis:1953,hastings:1970}: given the current value $X_n$ of the chain $\{X_j,\, j \geq 0 \}$,
 it consists in proposing a move $Y_{n+1}$ under a proposal
distribution $Q(X_n,\cdot)$. This move is then accepted with probability
\begin{align*}
 \alpha_n=1 \wedge \pi(Y_{n+1}) Q(Y_{n+1},X_n) / [\pi(X_n)
  Q(X_n,Y_{n+1})] \eqsp,
\end{align*}
where $a\wedge b$ stands for $\min(a,b)$; otherwise, $X_{n+1} = X_n$.

It is known that the
efficiency of MCMC methods depends upon the choice of the proposal distribution
\cite{rosenthal:2009}. For example, when sampling multi-modal
distributions, a Metropolis-Hastings algorithm with $Q(X_n,\cdot)$ equal to a
Gaussian distribution centered in $X_n$ tends to be stuck in one of the modes.
So the convergence of such an algorithm will be  slow, and the target distribution
will not be correctly approximated unless a huge number of points is sampled.

Efficient implementations of MCMC rely on a strong expertise of the user in
order to choose a proposal kernel and, more generally, design parameters
adapted to the target $\pi$.

This is the reason why {\em adaptive} and {\em interacting} MCMC methods have been introduced. Adaptive MCMC methods consist in choosing, at each iteration, a transition kernel $P_\theta$ among a family $\{P_\theta, \theta \in \Theta \}$ of kernels with invariant distribution $\pi$: the conditional distribution of $X_{n+1}$ given the past is $P_{\theta_n}(X_n, \cdot)$ where the parameter $\theta_n$ is chosen according to the past values of the chain $\{X_n, n\geq 0 \}$. From the pioneering Adaptive Metropolis algorithm of \cite{haario:saksman:tamminen:1999}, many adaptive MCMC have been proposed and successfully applied (see the survey papers by~\cite{andrieu:thoms:2008}, \cite{rosenthal:2009}, \cite{atchade:fort:moulines:priouret:2011} for example).

Interacting MCMC methods rely on the (parallel) construction of a family of
processes with distinct stationary distributions; the key behind these
techniques is to allow interactions when sampling these different processes.
At least one of these processes has $\pi$ as stationary distribution. The
stationary distributions of the auxiliary processes are chosen in such a way
that they have nice convergence properties, hoping that the process under study
will inherit them. For example, in order to sample multi-modal distributions, a
solution is to draw auxiliary processes with target distributions equal - up to
the normalizing constant - to tempered versions $\pi^{1/T_i}$, $T_i>1$.  
This solution is the basis of the parallel {\em tempering
  algorithm}~\cite{geyer:1991}, where the states of two parallel chains are
allowed to swap.  Following this tempering idea, different interacting MCMC
algorithms have been proposed and studied so
far~\cite{andrieu:jasra:doucet:delmoral:2007,bercu:delmoral:doucet:2009,delmoral:doucet:2010,brockwell:delmoral:doucet:2010}.

The {\em Equi-Energy sampler} of Kou, Zhou and Wong
\cite{kou:zhou:wong:2006} is an example of such interacting MCMC
algorithms.  $K$ processes are sampled in parallel, with target distributions
(proportional to) $\pi^{\beta_k}$, $1 = \beta_K > \beta_{K-1}
  >\cdots > \beta_1$. The first chain $Y^{(1)} = \{Y_n^{(1)},
  n\geq 0\}$ is usually a Markov chain; then $Y^{(k)}$ is
built from $Y^{(k-1)}$ as follows: with a fixed probability
$\InterractionProba$, the current state $Y_n^{(k)}$ is allowed
to jump onto a past state of the auxiliary chain
$\{Y^{(k-1)}_\ell, \ell \leq n \}$, and with probability
$(1-\InterractionProba)$, $Y_n^{(k)}$ is obtained using a
``local" MCMC move (such as a random walk Metropolis step or a
Metropolis-adjusted Langevin step). This mechanism includes the computation of
an acceptance ratio so that the chain $Y^{(k)}$ will have
$\pi^{\beta_{k}}$ as target density.  As the acceptance
probability of such a jump could be very low, only jumps toward selected past
values of $Y^{(k-1)}$, namely those with an {\em energy} close
to that of the current state $Y_n^{(k)}$, are allowed. This
selection step allows higher acceptance rates of the jump, and a faster
convergence of the algorithm is expected.

The Equi-Energy sampler has many design parameters: the interacting probability
$\InterractionProba$, the number $K$ of parallel chains, the temperatures
$T_k=1/\beta_k$, $k \in \{1,\dots,K\}$ and the selection function. It is known that all
of these design parameters play a role on the efficiency of the algorithm.
\cite{kou:zhou:wong:2006} suggest some values for all these parameters,
designed for practical implementation and based on empirical results on some
simple models. \cite{andrieu:jasra:doucet:delmoral:2011} discuss the choice of
the interacting probability $\InterractionProba$ in similar contexts;
\cite{atchade:roberts:rosenthal:2011} discuss the choice of the temperatures
$T_k$ of the chains for the Parallel Tempering algorithm.  Recently, an
algorithm combining parallel tempering with equi-energy moves have been
proposed by \cite{bargatti:grimaud:pommeret:2012}.

In this paper, we discuss the choice of the energy rings and the selection
function, when the jump probability $\InterractionProba$, the number $K$ of
auxiliary processes and the temperatures are fixed.  We introduce a new
algorithm, called {\em Adaptive Equi-Energy sampler} in which the selection
function is defined adaptively based on the past history of the sampler.  We
also address the convergence properties of this new sampler.  

Different kinds of convergence of adaptive MCMC methods have been addressed in
the literature: convergence of the marginals, the law of large numbers (LLN)
and central limit theorems (CLT) for additive functionals (see e.g.
\cite{roberts:rosenthal:2007} for convergence of the marginals and weak LLN of
general adaptive MCMC, \cite{andrieu:moulines:2006} or
\cite{saksman:vihola:2010} for LLN and CLT for adaptive Metropolis algorithms,
\cite{fort:moulines:priouret:2010} and
\cite{fort:moulines:priouret:vandekerkhove:2011} for convergence of the
marginals, LLN and CLT for general adaptive MCMC algorithms - see also the
survey paper by \cite{atchade:fort:moulines:priouret:2011}).

There are quite few analysis of the convergence of interacting MCMC samplers.
The original proof of the convergence of the Equi-Energy sampler in
\cite{kou:zhou:wong:2006} (resp. \cite{atchade:liu:2006}) contains a serious
gap, mentioned in \cite{atchade:liu:2006} (resp.
\cite{andrieu:jasra:doucet:delmoral:2008b}).
\cite{andrieu:jasra:doucet:delmoral:2011} established a strong LLN of a
simplified version of the Equi-Energy sampler, in which the number of levels is
set to $K=2$ and the proposal during the interaction step are drawn uniformly
at random in the past of the auxiliary process. Finally, Fort, Moulines and
Priouret \cite{fort:moulines:priouret:2010} established the convergence of
the marginals and a strong LLN for the same simplified version of the
Equi-Energy sampler (with no selection) but have removed the limitations on the
number of parallel chains.

The paper addresses the convergence of an interacting MCMC sampler in which the
proposal are selected from energy rings which are constructed adaptively at
each levels. In this paper, we obtain the convergence of the marginals and a
strong LLN of a smooth version of the Equi-Energy sampler and its adaptive
variant. We illustrate our results in several difficult scenarios such as
sampling mixture models with ``well-separated" modes and motif sampling in
biological sequences. The paper is organized as follows: in
Section~\ref{presentation}, we derive our algorithm and set the notations that
are used throughout the paper. The convergence results are presented in
Section~\ref{main:results}. Finally, Section~\ref{illustration} is devoted to
the application to motif sampling in biological sequences. The proofs of the
results are postponed to the Appendix.


\section{Presentation of the algorithm} \label{presentation}

\subsection{Notations}
Let $(\X,\mathcal{X})$ be a measurable Polish state space and $P$ be a Markov
transition kernel on $(\X,\mathcal{X})$. $P$ operates on bounded functions $f$
on $\X$ and on finite positive measures $\mu$ on $\mathcal{X}$:
\begin{displaymath}
 Pf(x) = \int P(x,\rmd y) f(y), \qquad \mu P(A) = \int \mu(dx) P(x,A) \eqsp.
\end{displaymath}
The $n$-iterated transition kernel $P^n$, $n \geq 0$ is  defined by:
\begin{displaymath}
  P^n(x,A)=\int P^{n-1}(x,\rmd y)P(y,A)= \int P(x,\rmd y) P^{n-1}(y,A) \eqsp;
\end{displaymath}
by convention, $P^0(x,A)$ is the identity kernel.  For a function $V :
\mathbf{X} \rightarrow [1,+\infty[$, we denote by $|f|_V$ the V-norm of a
function $f : \X \rightarrow \mathbb{R}$:
\begin{displaymath} |f|_V = \sup_{x
    \in \mathbf{X}} \frac{|f(x)|}{V(x)} \eqsp.
\end{displaymath}
If $V=1$, this norm is the usual uniform norm. Let $\mathcal{L}_V = \{ f :
\mathbf{X} \rightarrow \mathbb{R}, |f|_V < + \infty \}$.  We also define the
V-distance between two probability measures $\mu_1$ and $\mu_2$ by:
\begin{displaymath}
  \|\mu_1 - \mu_2 \|_V = \sup_{f,|f|_V\leq1} |\mu_1(f) - \mu_2(f) | \eqsp.
\end{displaymath}
When $V=1$, the V-distance is the total-variation distance and will be denoted by $\| \mu_1 - \mu_2 \|_{TV}$.

Let $(\Theta,\mathcal{T})$ be a measurable space, and $\{P_{\theta}, \theta \in
\Theta \}$ be a family of Markov transition kernels; $\Theta$ can be finite or
infinite dimensional. It is assumed that for all $A \in \mathcal{X}$,
$(x,\theta) \rightarrow P_{\T} (x,A)$ is $(\mathcal{X} \otimes \mathcal{T} |
\mathcal{B}([0,1]))$-measurable, where $\mathcal{B}([0,1])$ denotes the Borel
$\sigma$-field on $[0,1]$.

\subsection{The Equi-Energy sampler}

 \label{presentation:ees}
 Let $\pi$ be the probability density of the target distribution with respect
 to a dominating measure $\refmes$ on $(\X,\mathcal{X})$. In many applications, $\pi$ is known up to
 a multiplicative constant; therefore, we will denote by $\mesobs$ the
 (unnormalized) density.

 We denote by $P$ the Metropolis-Hastings kernel with proposal density kernel $q$ and invariant distribution $\pi$
 defined by:
\begin{equation*}
  P(x,A)= \int_A r(x,y) q(x,y) \refmes(\rmd y) + \mathbf{1}_A (x) \int
  (1-r(x,y)) q(x,y) \refmes(\rmd y) \eqsp,
\end{equation*}
where $(x,y) \mapsto r(x,y)$ is the acceptance ratio given by
\[
 r(x,y) = 1 \wedge \frac{\pi(y) q(y,x)}{\pi(x) q(x,y)} \eqsp.
\]

The Equi-Energy (EE) sampler proposed by
  \cite{kou:zhou:wong:2006} exploits the fact that it is often easier to
sample a tempered version $\pi^{\beta}$, $0<\beta<1$, of the target
distribution than $\pi$ itself. This is why the algorithm relies on an
auxiliary process $\{Y_n, n \geq 0\}$, run independently from $\{X_n\}$ and
admitting $\pi^{\beta}$ as stationary distribution (up to a normalizing
constant).  This mechanism can be repeated yielding to a multi-stages
Equi-Energy sampler.

We denote by $K$ the number of processes run in parallel. Let
$\InterractionProba \in (0,1)$.  Choose $K$ temperatures $T_1> \dots > T_K=1$
and set $\beta_k=1/T_k$; and $K$ MCMC kernels $\{P^{(k)}, 1\leq k \leq K \}$
such that $\pi^{\beta_k} P^{(k)} = \pi^{\beta_k}$.  $K$ processes $Y^{(k)} =
\{Y_n^{(k)}, n\geq 0\}$, $1 \leq k \leq K$, are defined by induction on the
probability space $(\Omega, \mathcal{F}, \PP)$.  The first auxiliary process
$Y^{(1)}$ is a Markov chain, with $P^{(1)}$ as transition kernel.  Given the
auxiliary process $Y^{(k-1)}$ up to time $n$, $\{Y_m^{(k-1)}, m \leq n\}$, and
the current state $Y_{n}^{(k)}$ of the process of level $k$, the Equi-Energy
sampler draws $Y_{n+1}^{(k)}$ as follows:
\begin{itemize}
\item (Metropolis-Hastings step) with probability $1-\InterractionProba$,
  $Y_{n+1}^{(k)} \sim P^{(k)}(Y_n^{(k)},\cdot)$.
\item (equi-energy step) with probability $\InterractionProba$, the algorithm
  selects a state $Z_{n+1}$ from the auxiliary process having an energy close
  to that of the current state. An acceptance-rejection ratio is then computed
  and if accepted, $Y_{n+1}^{(k)}=Z_{n+1}$; otherwise,
  $Y_{n+1}^{(k)}=Y_n^{(k)}$.
\end{itemize}
In practice, \cite{kou:zhou:wong:2006} only apply the equi-energy step when
there is at least one point in each ring. In \cite{kou:zhou:wong:2006}, the
distance between the energy of two states is defined as follows. Consider an
increasing sequence of positive real numbers
\begin{equation}
\label{eq:definition-bornea}
\bornea_0=0<\bornea_1<\dots<\bornea_S=+\infty \eqsp. 
\end{equation}
If the energies of two states $x$ and $y$ belong to the same energy ring, i.e.
if there exists $1\leq \ell \leq S$ such that $\bornea_{\ell-1} \leq
\mesobs(x),\mesobs(y) < \bornea_\ell$, then the two states are said to have
``close energy''.  The choice of the energy rings is most often a difficult
task. As shown in Figure~\ref{comparaison_10D}[right], the Equi-Energy sampler
is inefficient when the energy rings are not appropriately defined.
The efficiency of the sampler is increased when the variation
  of $\mesobs$ in each ring is small enough so that the equi-energy move is
  accepted with high probability.

\subsection{The Adaptive Equi-Energy sampler}
\label{sec:EEadaptive:algo}
We propose to modify the Equi-Energy sampler by adapting the energy rings ``on
the fly", based on the history of the algorithm.  Our new algorithm, so called
{\em Adaptive Equi-Energy} sampler (AEE) is similar to the Equi-Energy sampler
of \cite{kou:zhou:wong:2006} except for the equi-energy step, which relies on
adaptive boundaries of the rings. For the definition of the process $Y^{(k)}$,
$k \geq 2$, adaptive boundaries computed from the process $Y^{(k-1)}$ are used.

For a distribution $\T$ in $\Theta$, denote by $\bornea_{\T,\ell}$, $ \ell \in
\{1, \cdots, S-1 \}$ the bounds of the rings, computed from r.v.  with
distribution $\T$; by convention, $\bornea_{\T,0} = 0 \leq \bornea_{\T,1} \leq
\cdots \leq \bornea_{\T,S-1} \leq \bornea_{\T,S} = +\infty$.  Define the
associated energy rings $H_{\T,\ell} = [\bornea_{\T,\ell-1},
\bornea_{\T,\ell})$ for $\ell \in \{1, \cdots, S \}$. We consider selection
functions $g_{\T}(x,y)$ of the form
\begin{equation}
\label{eq:FonctionSelectionEE}
g_{\T}(x,y)  =  \sum_{\ell=1}^S h_{\T,\ell}(x) h_{\T, \ell}(y) \eqsp,  \quad  h_{\T, \ell}(x) = \left(1-d(\mesobs(x),\anneau_{\T,\ell}) \right)_+ \eqsp,
\end{equation}
where $d(\mesobs(x), \anneau_{\T,\ell})$ measures the distance between
$\mesobs(x)$ and the ring $\anneau_{\T, \ell}$.  By convention $h_{\T,\ell}=0$
if $\anneau_{\T,\ell} = \emptyset$.  We finally introduce a set of selection
kernels $\{K_\T^{(k)}, \T \in \Theta \}$ for all $k \in \{2, \cdots, K \}$
defined by
\begin{align} \label{eq:DefinitionKthetaAEE}
  K_{\theta}^{(k)}(x,A) &=\int_{A}
  \alpha_{\theta}^{(k)}(x,y)\frac{g_{\theta}(x,y)\theta(\rmd y)}{\int
    g_{\theta}(x,z)\theta(\rmd z)}
  +\mathbf{1}_A(x) \int \{1-\alpha_{\theta}^{(k)}(x,y)\} \frac{g_{\theta}(x,y)\theta(\rmd y)}{\int g_{\theta}(x,z)\theta(\rmd z)} \eqsp,
\end{align}
where
\begin{equation}
  \label{eq:DefinitionAlphaAEE}
  \alpha_{\theta}^{(k)}(x,y) =1 \wedge \left( \frac{\pi^{\beta_{k}-\beta_{k-1}}(y)
    \int g_{\theta} (x,z)
    \theta(\rmd z)}{\pi^{\beta_{k}-\beta_{k-1}}(x)\int g_{\theta}
    (y,z) \theta(\rmd z)} \right) \eqsp. \end{equation}
 $K_\T^{(k)}$ is associated to the equi-energy step when defining $Y^{(k)}$: a draw under  the selection kernel proportional to $ g_{\theta}(x,y)\theta(\rmd y)$ is combined with an acceptance-rejection step. The acceptance-rejection step is defined so that when $\T \propto \pi^{\beta_{k-1}}$, $\pi^{\beta_k}$ is invariant for $K_\T^{(k)}$~\cite{kou:zhou:wong:2006}.
 
 This equi-energy step is only allowed when each ring contains at least one
 point (of the auxiliary process $Y^{(k-1)}$ up to time $n$). We therefore
 introduce, for all positive integer $m$, the set $ {\tetaset}_m$:
\begin{equation}
\label{eq:ThetaM}
  {\tetaset}_m \eqdef \left\{\theta \in \Theta: \frac{1}{m} \leq \inf_x \int g_\theta(x,y)  \theta(\rmd y) \right\} \eqsp.
\end{equation}
 
With these notations, AEE satisfies for any $n \geq 0$ and $k \in \{1, \cdots,
K \}$,
\begin{equation}
  \label{eq:interactingMCMC}
  \PE[f(Y_{n+1}^{(k)})|\mathcal{F}_n^{(k)}] = \PE[f(Y_{n+1}^{(k)})|Y_n^{(k)}, Y_m^{(k-1)}, 1 \leq m \leq n]
= P_{\T_n^{(k-1)}}^{(k)}f(Y_n^{(k)}) \eqsp,
\end{equation}
where $\{\mathcal{F}_n^{(k)}, n\geq 0\}$ is the filtration defined by
$\mathcal{F}_n^{(k)}=\sigma \left( \left\{Y_m^{(l)}, 1\leq m \leq n, 1 \leq l
    \leq k \right\} \right)$; the transition kernel is given by $P_\T^{(1)} =
P^{(1)}$ and for $k \geq 2$,
\[
P_\T^{(k)} = (1-\InterractionProba \un_{\T \in \bigcup_{m \geq 1} \Theta_m}) P^{(k)}
+ \InterractionProba \un_{\T \in \bigcup_{m \geq 1} \Theta_m} \ K_{\T}^{(k)} \eqsp;
\]
and $\T_n^{(k)}$ is the empirical distribution
\begin{equation}
  \label{eq:DefinitionThetan}
  \T_n^{(k)} = \frac{1}{n} \sum_{m=1}^n \delta_{Y_m^{(k)}} \eqsp, \qquad k \in \{1, \cdots, K \}, n \geq 1 \eqsp.
\end{equation}
Different functions $d$ can be chosen. For example, the function given by
\begin{equation}
  \label{eq:hard-distance}
  d(\mesobs(x),\anneau_{\theta,\ell})= \un_{ \Rset \setminus \anneau_{\theta,\ell}}(\mesobs(x))= \left\{
    \begin{array}{ll}
0 & \text{if} \ \mesobs(x) \in \anneau_{\theta,\ell}, \\
1 & \text{otherwise}
    \end{array}
\right.
\end{equation}
yields to a selection function $g_\theta$ such that $g_\theta(x,y) = 1$ iff
$x,y$ are in the same energy ring and $g_\theta(x,y)=0$ otherwise.  In this
case, the acceptance-rejection ratio $\alpha_\theta^{(k)}(x,y)$ is equal to $1
\wedge (\pi^{\beta_{k}-\beta_{k-1}}(y) / \pi^{\beta_{k}-\beta_{k-1}}(x))$ upon
noting that by definition of the proposal kernel, the points $x$ and $y$ are in
the same energy ring.  By using this ``hard" distance during the equi-energy
jump, all the states of the auxiliary process having their energy in the same
ring as the energy of the current state are chosen with the same probability,
while the other auxiliary states have no chance to be selected.

Other functions  $d$ could be chosen, such as ``soft" selections of the form
\begin{equation}
  \label{eq:soft:selection:function}
  d(\mesobs(x),\anneau_{\theta,l})= \frac{1}{r} \min_{y \in \anneau_{\theta,\ell}} |\mesobs(x) -y|
\eqsp,
\end{equation}
where $r >0$ is fixed. With this ``soft'' distance, given a current state
$Y_n^{(k)}$, the probability for each auxiliary state $Y_i^{(k-1)}$, $i \leq
n$, to be chosen is proportional to
$g_{\theta_n^{(k-1)}}(Y_n^{(k)},Y_i^{(k-1)})$.  Then, the ``soft" selection
function allows auxiliary states having an energy in a $r$-neighborhood of the
energy ring of $\mesobs(Y_n^{(k)})$ to be chosen, as well as
states having their energy in this ring.  Nevertheless, this
selection function yields an acceptance-rejection ratio
$\alpha_\theta^{(k)}$ which may reveal to be quite costly to evaluate.

The asymptotic behavior of AEE will be addressed in Section~\ref{main:results}.
The intuition is that when the empirical distribution $\theta_n^{(k-1)}$ of the
auxiliary process of order $k-1$ converges (in some sense) to
$\theta_\star^{(k-1)}$, the process $\{Y_n^{(k)}, n\geq 0\}$ will behave (in
some sense) as a Markov chain with transition kernel
$P_{\theta_\star^{(k-1)}}^{(k)}$.

\subsection{A toy example (I)}

To highlight the interest of our algorithm, we consider toy examples: the
target density $\pi$ is a mixture of {$\Rset^d$}-valued
Gaussian\footnote{MATLAB codes for AEE are available at the
    address http://perso.telecom-paristech.fr/$\sim$schreck/index.html} . This
model is known to be difficult, as illustrated (for example) in
\cite{atchade:fort:moulines:priouret:2011} for a random walk
Metropolis-Hastings sampler (SRWM), an EE-sampler and a parallel tempering
algorithm. Indeed, if the modes are well separated, a Metropolis-Hastings
algorithm using only ``local moves" is likely to remain trapped in one of the
modes for a long-period of time. In the following, AEE is implemented with ring boundaries computed as described in Section~\ref{ring:construction}.

Figure~\ref{comparaison}.(a) displays the target density $\pi$ and the
simulated one for three different algorithms (SRWM, EE and AEE) in
one dimension. The histograms are obtained with $10^5$
samples; for EE and AEE, the probability of interaction is
$\InterractionProba=0.1$, the number of parallel chains is equal to $K=5$ and
the number of rings is $S=5$. For the adaptive definition of the rings in AEE,
we choose the ``hard" selection (\ref{eq:hard-distance}) and the construction
of the rings defined in Section~\ref{ring:construction}. In the same vein,
Figure~\ref{comparaison_2D} displays the points obtained by the three
algorithms when sampling a mixture of two Gaussian distributions in
two dimensions. As expected, in both figures, SRWM never
explores one of the modes, while EE and AEE are far more efficient.

\begin{figure}[ht!]
\begin{center}
  \includegraphics[height=5.1cm,width=.8\linewidth]{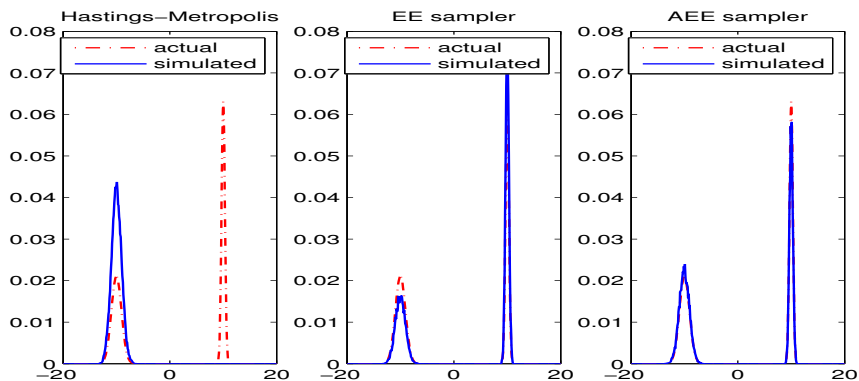}
\end{center}
    \caption{Comparison of SRWM (left), EE (center) and AEE (right) for a Gaussian mixture in one dimension} \label{comparaison}
\end{figure}

\begin{figure}[ht!]
\begin{center}
  \includegraphics[width=13cm]{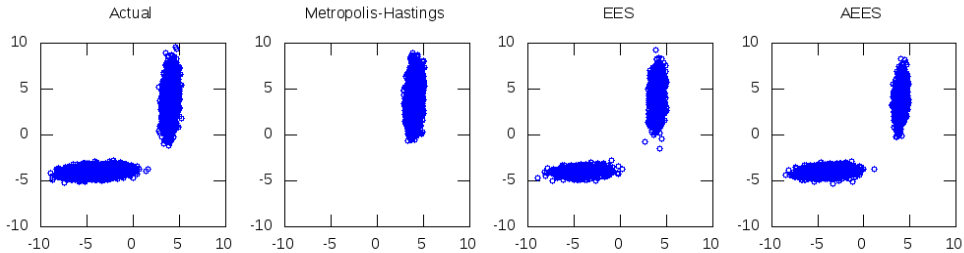}
\end{center}
\caption{Comparison of the algorithms for a Gaussian mixture in two dimensions: (from left to right) the true density, SRWM, EE and AEE.}
\label{comparaison_2D}
\end{figure}

To compare EE and AEE in a more challenging situation, we consider the case of
a mixture with two components in ten dimensions.  We run EE
and AEE with $K=3$ parallel chains with respective temperatures $T_1=1, T_2=9,
T_3=60$, the probability of jump $\InterractionProba$ is equal to $0.1$, and
the number of rings is $S=50$.  Both algorithms are initialized in one of the
two modes of the distribution. For the Metropolis-Hastings step, we use a
Symmetric Random Walk with Gaussian proposal; the covariance matrix of the
proposal is of the form $c \ I$ where $c$ is calibrated so that the mean
acceptance rate is approximatively $0.25$. Figure~\ref{comparaison_10D}
displays, for each algorithm, the $L^1$-norm of the empirical mean, averaged
over $10$ independent trajectories, as a function of the length of the chains.

In order to show that the efficiency of EE depends crucially upon the choice of
the rings, we choose a set of boundaries so that in practice, along one run of
the algorithm, some of the rings are never reached.
Figure~\ref{comparaison_10D}(a) compares EE and AEE in this extreme case: even
after $2 \times 10^5$ iterations, all of the equi-energy jumps
are rejected for the (non-adaptive) EE, and the algorithm is trapped in one of
the modes.  This does not occur for AEE, and the $L^1$-error tends to zero as
the number of iterations increases.  This illustrates that our adaptive
algorithm avoids the poor behaviors that EE can have when the choice of its
design parameters is inappropriate.

We now run EE in a less extreme situation: we choose (fixed) energy rings so
that the sampler can jump more easily than in the previous experiment between
the modes.  Figure~\ref{comparaison_10D}(b) illustrates that the adaptive
choice of the energy rings speeds up the convergence, as it makes the
equi-energy jumps be more often accepted. To have a numerical comparison, the
equi-energy jumps were accepted about ten times more often for AEE than for 
EE.

\begin{figure}[ht!]
\begin{center}
     \begin{tabular}{cc}
      \includegraphics[height=4.2cm,width=.45\linewidth]{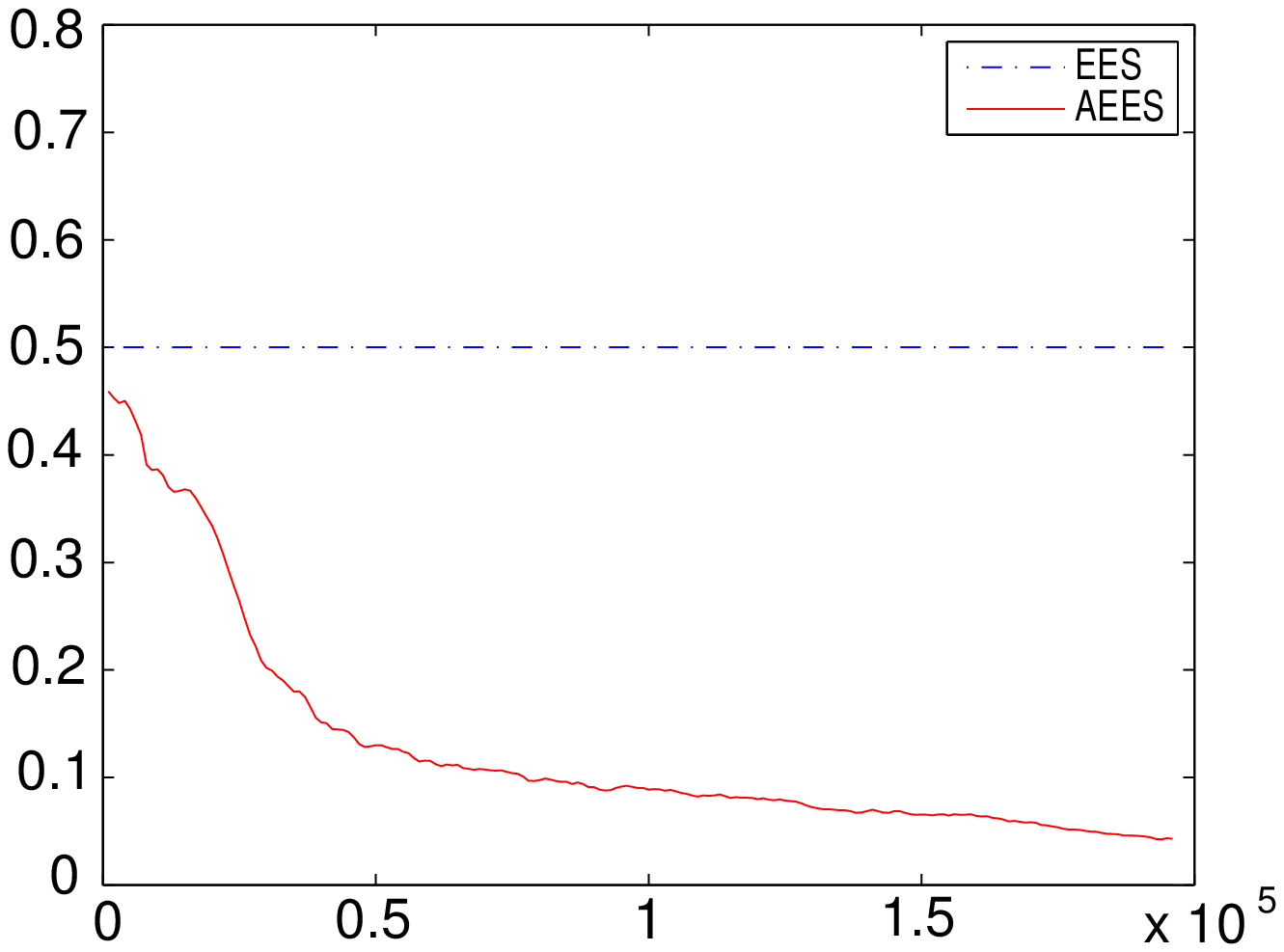} &
       \includegraphics[height=4.2cm,width=.45\linewidth]{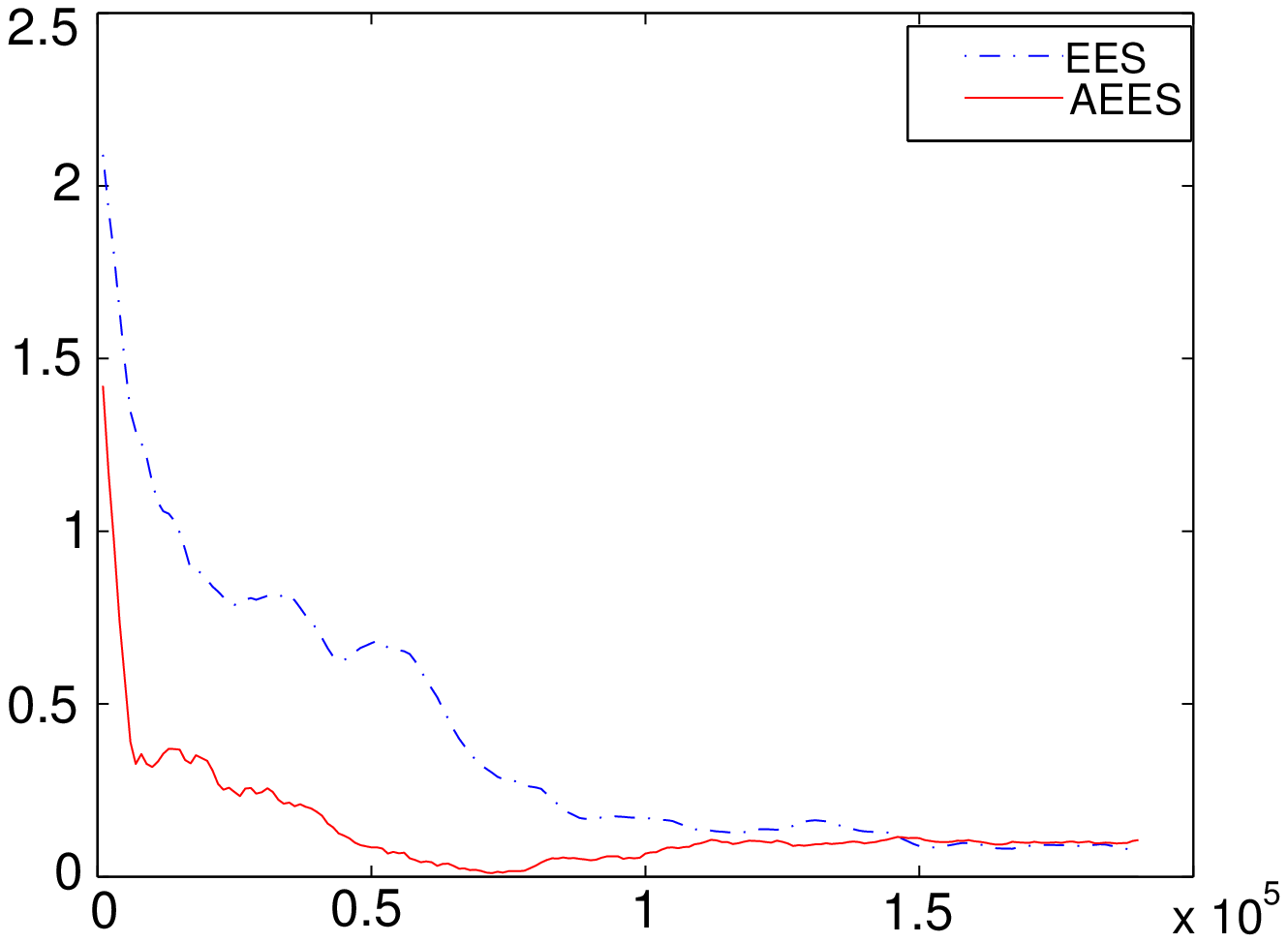} \\
     (a) & (b)\\
    \end{tabular}
\end{center}
    \caption{Error of EE (dashed line) and AEE for two different target densities in ten dimensions.} \label{comparaison_10D}
\end{figure}

\subsection{Toy example (II)}
\label{sec:toyexample}
For a better understanding on how our algorithm behaves,
Figure~\ref{erreur:epsilon:strates}.(a) displays the evolution of the ring
bounds used in the definition of $Y^{(K)}$. In this numerical application, the
target density is a mixture of two Gaussian distributions in one dimension; EE and
AEE are run with $K=5$ chains, $S=5$ rings and $\InterractionProba=0.1$, for a
number of iterations varying from $0$ to $10^5$. As expected, the ring bounds
become stable after a reasonable number of iterations. Moreover, we observed
that the (non-adaptive) EE run with the rings fixed to the
limiting values obtained with AEE behaves remarkably well.

Finally, to have an idea on the role played by $\InterractionProba$,
Figure~\ref{erreur:epsilon:strates}.(b) displays the average $L^1$ error of AEE
for a mixture of two Gaussian distributions in one dimension,
after $2 \times 10^5$ iterations and for 100 independent
trajectories when $\InterractionProba$ is varying from 0 to 1. If
$\InterractionProba$ is too small, AEE is not mixing well enough, and if
$\InterractionProba$ is too large, the algorithm jumps easily from one mode to
another but does not explore well enough each mode, which explains the `u'
shape of the curve. This experiment suggests that there exists
an optimal value for $\InterractionProba$, but to our best knowledge, the
optimal choice of this design parameter is an open problem.

\begin{figure}[ht!]
\begin{center}
     \begin{tabular}{cc}
      \includegraphics[height=3.3cm,width=.49\linewidth]{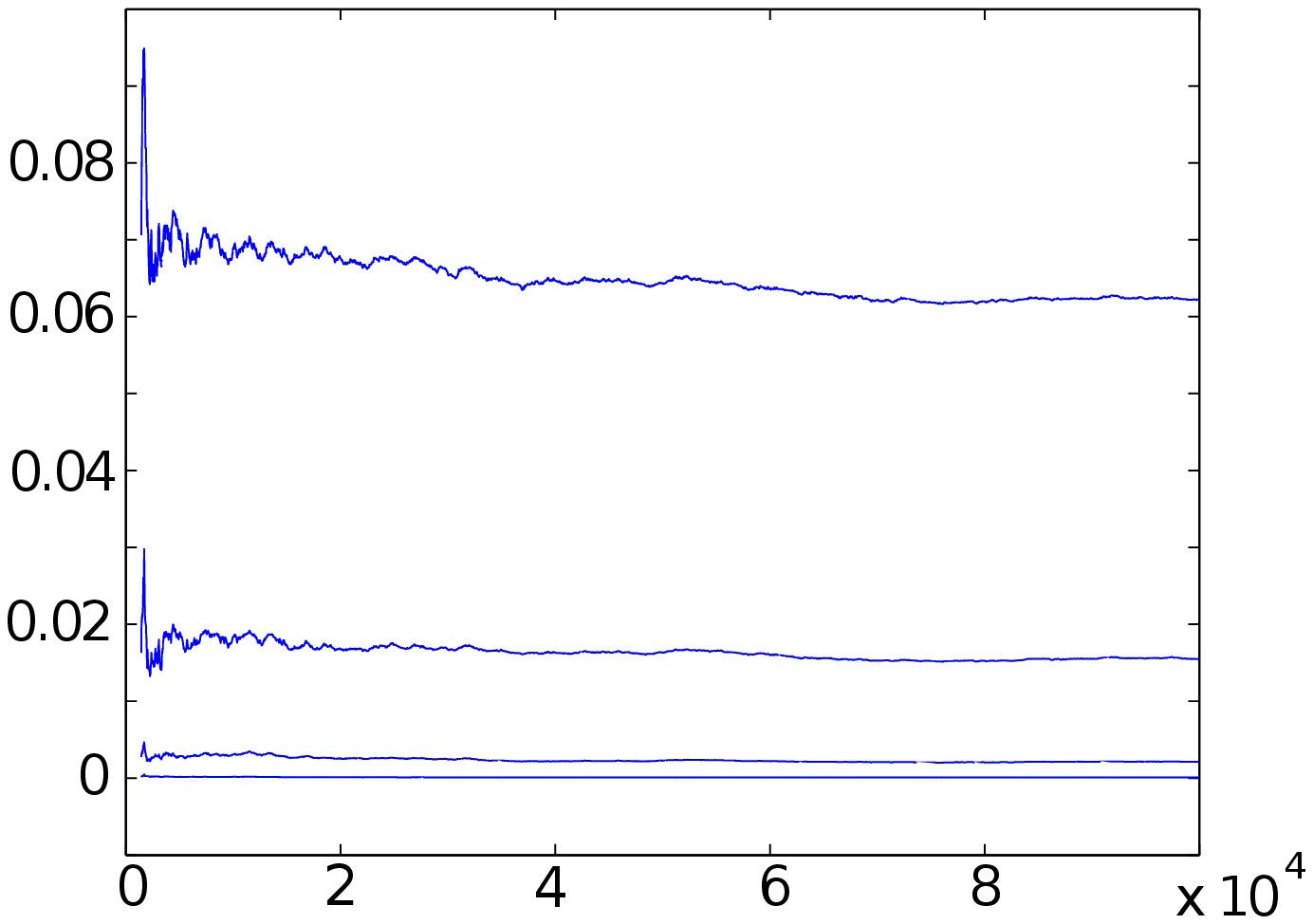} &
      \includegraphics[height=3.3cm,width=.4\linewidth]{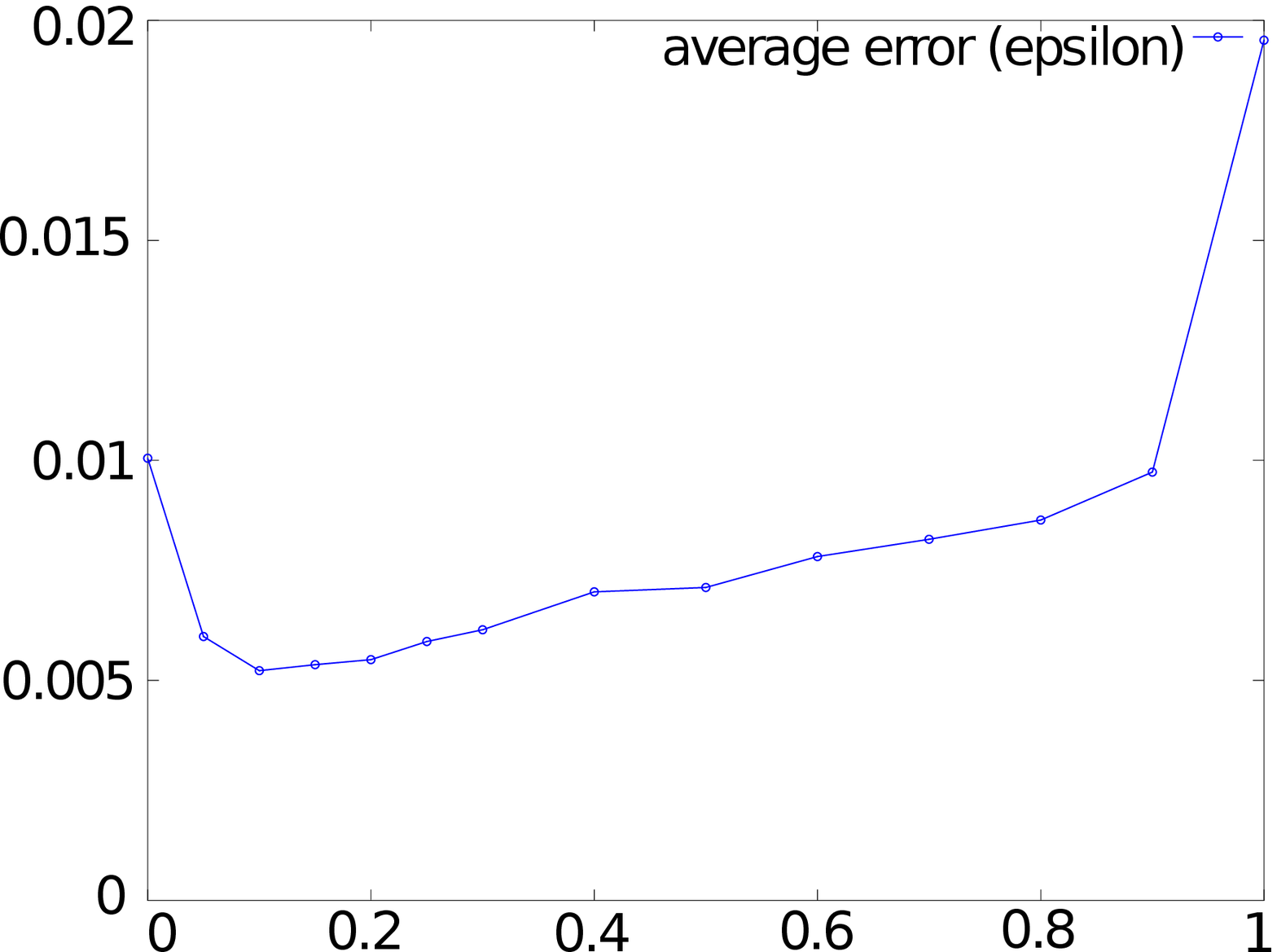} \\
      (a) & (b)\\
    \end{tabular}
\end{center}
    \caption{(a): Evolution of the ring bounds; (b): Averaged error of  AEE as a function of $\InterractionProba$.} \label{erreur:epsilon:strates}
\end{figure}


\section{Convergence of the Adaptive Equi-Energy sampler} \label{main:results}

In this section, the convergence of the $K$-stages Adaptive Equi-Energy sampler
is established.  In order to make the proof easier, we consider the case when
the distance function $d$ in the definition of the selection function
(\ref{eq:FonctionSelectionEE}) is given by~(\ref{eq:soft:selection:function}).

\cite{fort:moulines:priouret:2010} provide sufficient conditions for the
convergence of the marginals and the strong LLN (s-LLN) of interacting MCMC
samplers.  We use their results and show the convergence of the marginals i.e.
\[ \lim_{n \rightarrow \infty} \PE\left[f(Y_n^{(K)}) \right] =
\pi(f) \eqsp,
\]
for any continuous bounded functions $f$. Note that this
  implies that this limit holds for any indicator function $f =\un_A$ such that
  $\PP(\partial A)=0$ where $\partial A$ denotes the boundary
    of $A$ \cite[Theorem
  2.1]{billingsley:1968}. We then establish the s-LLN: for a wide class of
  continuous (un)bounded functions $f$, 
\[
\lim_{n \rightarrow \infty} \frac{1}{n} \sum_{m=0}^{n-1}
f(Y_m^{(K)}) = \pi(f) \eqsp, \PP-\mathrm{a.s.}
\]

\subsection{Assumptions}
\label{subsec:hyp}
Our results are established for target distributions $\pi$ satisfying
\begin{hypEE}
\label{EE1}
 \begin{enumerate}[(a)]
 \item \label{EE1:a} $\pi$ is the density of a probability distribution
   on the measurable Polish space $(\X, \Xsigma)$ and
     $\sup_{\X}\pi<\infty$ and for any $s \in (0,1]$, $\int \pi^s(x) \ \rmd x < \infty$.
  \item \label{EE1:b} $\pi$ is continuous and positive on $\X$.
 \end{enumerate}
\end{hypEE}
Usually, the user knows $\pi$ up to a normalizing constant: hereafter,
$\mesobs$ will denote this available (unnormalized) density.

As in \cite{fort:moulines:priouret:2010}, we first introduce a set of
conditions that will imply the geometric ergodicity of the kernels
$P_{\theta}^{(k)}$, and the existence of an invariant probability measure for
$P_{\theta}^{(k)}$ (see conditions E\ref{EE3}).  We finally introduce
conditions on the boundaries of the adaptive energy rings (see conditions
E\ref{EE:ring}).  Examples of boundaries satisfying E\ref{EE:ring} and computed
from quantile estimators are given in Section~\ref{ring:construction}
(see also \cite{schreck:etal:2012} for stochastic
  approximation-based adapted boundaries).

Convergence of adaptive and interacting MCMC samplers is addressed in the
literature by assuming containment conditions and diminishing adaptations (so
called after \cite{roberts:rosenthal:2007}). Assumptions E\ref{EE3} is the
main tool to establish a (generalized) containment condition. In our
algorithm, the adaptation mechanism is due to (a) the interaction with an
auxiliary process and (b) the adaption of the rings.  Therefore, assumptions
E\ref{EE3} and E\ref{EE:ring} are related to the diminishing adaptation
condition (see e.g.  Lemma~\ref{lemme:cv:dv:proba} in
Section~\ref{secproof:ergo_ees}).

\begin{hypEE}
\label{EE3}
For each $k \in \{1, \dots, K \}$:
 \begin{enumerate}[(a)]
 \item \label{EE3:a} $P^{(k)}$ is a $\phi$-irreducible transition kernel which is
   Feller on $(\mathbf{X},\mathcal{X})$ and such that $\pi^{\beta_k} P^{(k)}=\pi^{\beta_k}$.
 \item \label{EE3:b} There exist $\lambda_k \in (0,1)$, $b_k<+\infty$ and
   $\tau_k \in \left( 0, \tau_{k-1} \beta_{k-1} / \beta_k \right)$ such that $
   P^{(k)}W_k \leq \lambda_k W_k + b_k$ with
\begin{equation}
  \label{eq:DefinitionW}
   W_k(x) =  \left( \frac{\pi^{\beta_k}(x)}{\sup_{\X} \pi^{\beta_k}}\right)^{-\tau_k} \eqsp;
\end{equation}
by convention, $\tau_0 \beta_0 =\beta_1$.
\item \label{EE3:c} For all $p \in (0,\sup_ {\mathbf{X}} \pi)$, the sets
  $\{\pi \geq p\}$ are 1-small for $P^{(k)}$.
 \end{enumerate}
\end{hypEE}
Note that by definition of $\tau_k$ and E\ref{EE1}\ref{EE1:a}, $W_{k+1} \in
\mathcal{L}_{W_k}$ and $\int W_{k}(x) \pi^{\beta_{k}}(x) \rmd x < \infty$.

E\ref{EE3} is satisfied for example if for each $k$, $P^{(k)}$
  is a symmetric random walk Metropolis Hastings kernel; and $\pi$ is a
  sub-exponential target density~\cite{roberts:tweedie:1996,jarner:hansen:2000}.

In our algorithm, $Y^{(1)}$ is a Markov chain with transition kernel $P^{(1)}$.
As discussed in \cite{meyn:tweedie:1993}[chapters 13 and 17], E\ref{EE3} is sufficient
to prove ergodicity and a s-LLN for $Y^{(1)}$. E\ref{EE3} also implies uniform
$W_1$-moments for $Y^{(1)}$. These results, which initializes our proof by
recurrence of the convergence for the process number $K$, is given in
Proposition~\ref{prop:lgn:first}. Define the probability distributions
\begin{equation}
  \label{eq:ThetaStar}
  \theta_\star^{(k)}(\rmd x) = \frac{\pi^{\beta_{k}}(x)}{\int \pi^{\beta_{k}}(z) \refmes(\rmd
  z)}\refmes(\rmd x)\eqsp, \qquad k \in \{1, \cdots, K \} \eqsp.
\end{equation}

 \begin{proposition}
\label{prop:lgn:first}
Assume E\ref{EE1}\ref{EE1:a}, E\ref{EE3} and $\PE\left[W_1(Y_0^{(1)})
\right] < \infty$. Then,
  \begin{enumerate}[(a)]
  \item \label{EE2:c} For all bounded measurable functions $f$, $\lim_{n \to \infty}
    \PE\left[f(Y_n^{(1)})\right] = \T_\star^{(1)}(f)$.
  \item \label{EE2:a} $\theta_{\star}^{(1)}(W_2) < +\infty$, and for any
    measurable function $f$ in $\mathcal{L}_{W_1}$, $\lim_{n \to \infty}
    \theta_n^{(1)}(f) = \theta_{\star}^{(1)}(f)$ a.s.
  \item \label{EE2:b} $\sup_n \PE \left[W_1\left(Y_n^{(1)} \right) \right] < \infty$.
  \end{enumerate}
 \end{proposition}

\begin{hypEE}  \label{EE:ring}
  \begin{enumerate}[(a)]
  \item \label{ring:c} For any $k \in \{1, \dots, K-1 \}$, $\inf_{\ell \in \{1,
      \cdots, S-1 \} } \int h_{\theta_\star^{(k)},\ell}(y) \ 
    \theta_\star^{(k)}(\rmd y) >0$.
  \item \label{ring:a} For any $k \in \{1, \dots, K-1\}$ and $\ell \in \{1,
    \cdots, S-1 \}$, $\lim_{n \to \infty} \left| \bornea_{\theta_n^{(k)}, \ell}
      - \xi_{\theta_\star^{(k)}, \ell} \right| =0$ w.p.1
\item \label{ring:b} There exists $\vitborn>0$ such that for any $k \in \{1,
  \dots, K-1\}$, any $\ell \in \{1, \cdots, S-1 \}$, and any $\vit \in
  (0,\vitborn)$, $\limsup_n \ n^{\vit} \ \left| \bornea_{\theta_{n+1}^{(k)},
      \ell} - \xi_{\theta_n^{(k)}, \ell} \right| < \infty$ w.p.1.
 \end{enumerate}
\end{hypEE}
Note that by definition of $h_{\T,\ell}$ (see (\ref{eq:FonctionSelectionEE}))
\begin{equation}
  \label{eq:CS:EE:ringa}
  \int h_{\theta,\ell}(y) \ \theta(\rmd y) \geq \theta (\{y: \pi_u(y) \in
H_{\theta,\ell} \}) \eqsp.
\end{equation}
Condition E\ref{EE:ring}\ref{ring:a} states that the rings
$\{H_{\theta_n^{(k)},\ell)}, n\geq 0 \}$ converge to
$H_{\theta_\star^{(k)},\ell}$ w.p.1; therefore, E\ref{EE:ring}\ref{ring:c} is
satisfied as soon as the limiting rings are of positive probability under the
distribution of $\pi_u(Z)$ when $Z \sim \theta_\star^{(k)}$.

When the energy bounds are fixed, the conditions
  E\ref{EE:ring}\ref{ring:a}-\ref{ring:b} are clearly satisfied and
  E\ref{EE:ring}\ref{ring:c} holds under convenient choice of the rings. We
  will discuss in Section~\ref{ring:construction} how to check the condition
  E\ref{EE:ring} with adaptive energy bounds.

\subsection{Convergence results}
\label{convergence:aees}
Proposition~\ref{drift:small:aees} shows that the kernels $P_\theta^{(k)}$
satisfy a geometric drift inequality and a minorization condition, with
constants in the drift independent of $\theta$ for $\theta\in \tetaset_m$ ($\tetaset_m$ being defined in (\ref{eq:ThetaM})). The
proof is in Appendix~\ref{secproof:drift:small:aees}.
\begin{proposition} \label{drift:small:aees} Assume E\ref{EE1}\ref{EE1:a} and E\ref{EE3}.
  For all $k \in \{1, \dots, K \}$:
  \begin{enumerate}[(a)]
  \item There exist $\tilde{\lambda}_k \in (0,1)$ and $\tilde{b}_k < + \infty$
    such that for all $m \geq 1$ and any $\theta \in \tetaset_{m}$,
\begin{equation}
\label{eq:drift}
   P_{\theta}^{(k)}W_k \leq \tilde{\lambda}_k W_k +\tilde{b}_k \ m \,  \theta (W_k) \ \eqsp.
\end{equation}
For all $p \in (0,\sup_{\mathbf{X}} \pi)$ and all $\theta \in \bigcup_m
\tetaset_m$, the sets $\{\pi\geq p \}$ are $1$-small for
$P_{\theta}^{(k)}$ and the minorization constants depend neither upon
$\theta$ nor on $m$.
\item \label{drift:small:aees:item2} For all $\T \in \bigcup_m \tetaset_{m}$,
  there exists a probability measure $\pi_{\T}^{(k)}$ invariant for
  $P_{\T}^{(k)}$.  In addition, $\pi_{\theta}^{(k)}(W_k) \leq \tilde{b}_k
  (1-\tilde {\lambda}_k)^{-1} \ m \theta(W_k) $ for $\T \in \tetaset_m$.
  \end{enumerate}
\end{proposition}

Theorem~\ref{theo:tout} is proved in Section~\ref{proof}.
Theorem~\ref{theo:tout}\eqref{prop:subsetTheta}  shows that there exists
$m_\star \geq 1$ such that w.p.1, for all $n$ large enough $\theta_n^{(k)}$
belongs to some $\tetaset_{m_\star}$.  Note that in
\cite{andrieu:jasra:doucet:delmoral:2008b}, a s-LLN for the Equi-Energy
sampler is established by assuming that there exists a deterministic positive
integer $m$ such that w.p.1, $\theta_n^{(k)} \in \tetaset_m$ for any $n$.  Such
a condition is quite strong since roughly speaking, it means that after $n$
steps (even for small $n$), all the rings contain a number of point which is
proportional to $n$, w.p.1. This is all the more difficult to guarantee in
practice, that the rings have to be chosen prior to any exploration of $\pi$.
Our approach allows to relax this strong condition.

The convergence of the marginals and the law of large numbers both require the
convergence in $n$ ($k$ fixed) of $\{\pi_{\theta_n^{(k)}}^{(k+1)}(f), n \geq 0
\}$ for some functions $f$. Such a convergence is addressed in
Theorem~\ref{theo:tout}\eqref{prop:cvg:pi}.  We will then have the main
ingredients to establish the convergence results for the processes $Y^{(k)}$,
$k \geq 1$.
\begin{theorem}
  \label{theo:tout} Assume E\ref{EE1}, E\ref{EE3}, E\ref{EE:ring} and $\PE[W_k(Y_0^{(k)})] <
  \infty$ for all $k \in \{1, \cdots, K\}$.
  \begin{enumerate}[(a)]
  \item \label{prop:subsetTheta} There exists $m_\star \geq 1$ such that for
    all $k \in \{1, \dots, K-1 \}$
\begin{align} \label{proba:teta:u}
  \PP\left( \bigcup_{q \geq 1} \bigcap_{n \geq q} \{\theta_n^{(k)} \in
    \tetaset_{m_\star} \} \right)=1 \eqsp.
\end{align}
\item \label{prop:cvg:pi} For any $k \in \{1, \cdots, K \}$, any $a \in (0,1)$
  and any continuous function $f \in \mathcal{L}_{W_{k}^a}$,
\[
\lim_{n \to \infty} \pi_{\theta_n^{(k-1)}}^{(k)}(f) = \theta_{\star}^{(k)}(f)
\eqsp, \ \text{w.p.1} \eqsp.
\]
\item For any $k \in \{1, \cdots, K\}$ and for all bounded continuous function
  $f: \X \to \Rset$, $\lim_{n \to \infty} \mathbb{E}[f(Y_n^{(k)})] =
  \theta_{\star}^{(k)}(f)$.
\item Let $a \in (0,\frac{1+\vitborn}{2} \wedge 1)$. For any $k \in \{1,
  \cdots, K\}$ and for all continuous function $f$ in $\mathcal{L}_{W_{k}^a}$
\begin{align*}
  \lim_{n \to \infty} \frac{1}{n} \sum_{m=1}^n f(Y_m^{(k)}) = \theta_{\star}^{(k)}(f) \qquad
  \PP-\text{a.s.} \eqsp.
\end{align*}
  \end{enumerate}
\end{theorem}

Observe that, for the process $\{Y^{(k)}, k \in \Nset\}$, the family  of functions for which   the law of large numbers holds depends  \textit{(i)} upon $\Gamma$ given by EE\ref{EE:ring}(\ref{ring:b}) i.e. in some sense, depends upon the adaptation rate; and \textit{(ii)} the temperature ladder.  In the case $\tau_k$ can be chosen arbitrarily close to $ \beta_{1}/\beta_k$ for any $k$ (see comments after \cite[Theorem 4.1 and 4.3]{jarner:hansen:2000}),  this family of functions only depends upon $\Gamma$ and the lowest inverse temperature : it is all the more restrictive than $\beta_1$ is small.

To our best knowledge, we are the first to prove such convergence results for
AEE (and EE): previous works
\cite{fort:moulines:priouret:2010,andrieu:jasra:doucet:delmoral:2011} consider
the simpler case when there is no selection i.e. $g_\theta(x,y)=1$.

\subsection{Comments on Assumption E\ref{EE:ring}}
\label{ring:construction} 
We propose to choose the adaptive boundaries $\bornea_{\T, \ell}$ as the
$p_\ell$-quantile of the distribution of $\mesobs(Z)$ when $Z$ is sampled under
the distribution $\T$.  This section proves that empirical
  quantiles of regularly spaced orders are examples of adaptive boundaries
  $\bornea_{\theta_n^{(k)},\ell}$ satisfying E\ref{EE:ring}. Let $F_{\T}$ be
the cumulative distribution function (cdf) of the r.v.  $\mesobs(Z)$ when $Z
\sim \T$:
\begin{align*}
  F_{\T}(x) = \int \un_{\{ \mesobs(z) \leq x \}} \T(\rmd z) \eqsp, \qquad x \in [0, \infty)\eqsp.
\end{align*}
We denote the quantile function associated to $\mesobs(Z)$ by:
\begin{align*}
  F_{\T}^{-1}(p) = \inf \{x \geq 0, F_{\T}(x) \geq p\} \quad \forall
  p>0 \eqsp; \qquad \qquad F_{\T}^{-1}(0)=0 \eqsp.
\end{align*}
With this definition, for $0 < p_1 < \cdots < p_{S-1} <1$, we set
$\bornea_{\T,\ell} \eqdef F_{\T}^{-1} \left( p_\ell \right)$.

With this choice of the boundaries, the condition
  E\ref{EE:ring}\ref{ring:c} holds: by (\ref{eq:CS:EE:ringa}),
  E\ref{EE:ring}\ref{ring:c} is satisfied because $\pi$ is continuous. The
conditions E\ref{EE:ring}\ref{ring:a}-\ref{ring:b} require the convergence of
the quantile estimators and a rate of convergence of the variation of two
successive boundaries. To prove such conditions, we use an Hoeffding-type
inequality.  
\begin{proposition} \label{prop:vitesse:quantiles}
  Assume
  \begin{enumerate}[(i)]
  \item  \label{hyp:fd}    The cumulative distribution function $F_{\T_{\star}^{(1)}}$ where
      $\theta_\star^{(1)}$ is given by (\ref{eq:ThetaStar}), is differentiable
      with positive derivative on $F_{\theta_\star^{(1)}}^{-1}((0,1))$.
    \item there exists $\overline{W}$ such that $Y^{(1)}$ is a
      $\overline{W}$-uniformly ergodic Markov chain with initial distribution
      satisfying $\PE\left[Y^{(1)}_0 \right] < \infty$.
  \end{enumerate}
 Then E\ref{EE:ring}\ref{ring:a}-\ref{ring:b} hold with $\vitborn = 1/2$ and $K=2$.
\end{proposition}

The proof is in Section~\ref{secproof:prop:vitesse:quantiles}. Extensions of
Proposition~\ref{prop:vitesse:quantiles} to the case when $Y^{(1)}$ is not a
uniformly ergodic Markov chain is, to our best knowledge, an open question.
Therefore, our convergence result of AEE when the boundaries are the quantiles
defined by inversion of the cdf of the auxiliary process applies to the
$2$-stage level and seems difficult to extend to the $K$-stage, $K>2$.

We proved recently in \cite{schreck:etal:2012} that when the
  quantiles are defined by a stochastic approximation procedure, the conditions
  E\ref{EE:ring}\ref{ring:a}-\ref{ring:b} hold even under very weak conditions
  on the auxiliary $Y^{(k)}$, $k \geq 2$. In this case, the convergence of the
  $K$-level AEE with $K>2$ is established.


\section{Application to motif sampling in biological sequences} \label{illustration}
One of the challenges in biology is to understand how gene expression is
regulated. Biologists have found that proteins called transcription factors
play a role in this regulation. Indeed, transcription factors bind on special
motifs of DNA and then attract or repulse the enzymes that are responsible of
transcription of DNA sequences into proteins. This is the reason why finding
these binding motifs is crucial. But binding motifs do not contain
deterministic start and stop codons: they are only random sequences that occurs
more frequently than expected under the background model.

Several methods have been proposed so far to retrieve binding
motifs~\cite{stormo:hartzell:1989,lawrence:reilly:1990,bailey:elkan:1994},
which yields to a complete Bayesian
model~\cite{liu:neuwald:lawrence:1995}.  Among the Bayesian approach, one
effective method is based on the Gibbs sampler~\cite{lawrence:et:al:1993} - it
has been popularized by software programs
\cite{liu:brutlag:liu:2001,roth:hughes:estep:church:1998}.  Nevertheless, as
discussed in \cite{kou:zhou:wong:2006}, it may happen that classical MCMC
algorithms are inefficient for this Bayesian approach.
Therefore, \cite{kou:zhou:wong:2006} show the interest of the Equi-Energy
sampler when applied to this Bayesian inverse problem; more recently,
\cite{woodard:rosenthal:2011} proposed a Gibbs-based algorithm for a similar
model (their model differs from the following one through the assumptions on
the background sequence).

We start with a description of our model for motif sampling in biological
sequences - this section is close to the description in
~\cite{kou:zhou:wong:2006} but is provided to make this paper self-contained.
We then apply AEE and compare it to the Interacting MCMC of \cite[Section
3]{fort:moulines:priouret:2010} (hereafter called I-MCMC), and to a
Metropolis-Hastings algorithm (MH). Comparison with Gibbs-based algorithms
(namely BioProspector and AlignACE) can be found in the paper of
\cite{kou:zhou:wong:2006}.

The available data is a DNA sequence, which is modeled by a background sequence
in which some motifs are inserted. The background sequence is represented by a
vector $\mathsf{S}=(s_1, s_2, \dots, s_L)$ of length $L$. Each element $s_i$ is
a nucleotide in $\{A,C,G,T\}$; in this paper, we will choose
the convention $s_i \in \{1,2,3,4\}$.  The length $w$ of a motif is assumed to
be known. The motif positions are collected in a vector $A=(a_1, \dots, a_L)$,
with the convention that $a_i =j$ iff the nucleotide $s_i$ is located at
position number $j$ of a motif; and $a_i =0$ iff $s_i$ is not in the motif. The
goal of the statistical analysis of the data $\mathsf{S}$ is to explore the
distribution of $A$ given the sequence $\mathsf{S}$. We now introduce notations
and assumptions on the model in order to define this conditional distribution.

We denote by $p_0$ the probability that a sub-sequence of length $w$ of
$\mathsf{S}$ is a motif. It is assumed that the background sequence is a Markov
chain with (deterministic) transition matrix $\param_0 =\{\param_0(i,j)\}_{1
  \leq i,j \leq 4}$ on $\{1, \cdots, 4 \}$; and the nucleotide in a sequence
are sampled from a multinomial distribution of parameter $\param=
\{\param(i,j)\}_{1 \leq i \leq 4, 1 \leq j \leq w}$ , $\param(i,j)$ being the
probability for the $j$-th element of a motif to be equal to $i$.

In practice, it has been observed that approximating $\param_0(i,j)$ by the
frequency of jumps from $i$ to $j$ in the (whole) sequence $\mathsf{S}$ is
satisfying. It is assumed that the r.v. $(\param, p_0)$ are independent with
prior distribution $\prod_{j=1}^w \prior(\param(\cdot,j))$ and $\prior'(p_0)$;
$\prior(\param(\cdot, j))$ is a Dirichlet distribution with parameters
$\iota_j=(\iota_{j,1}, \cdots, \iota_{j,4})$ and $\prior'(p_0)$ is a Beta
distribution with parameters $(b_1,b_2)$. $\iota_j$, $b_1$ and $b_2$ are
assumed to be known.

Therefore, given $(\param, p_0)$, $(A,\mathsf{S})$ is a Markov chain described
as follows:
\begin{itemize}  
\item[$\bullet$] If $a_{k-1} \in \{1,\dots,w-1\}$ then $a_k=a_{k-1}+1$; else
  $\PP(a_k =1 \vert a_{k-1} \in \{0, w\}, p_0, \param) = 1- \PP(a_k=0\vert a_{k-1} \in \{0, w\}, p_0, \param) =p_0$.
\item[$\bullet$] If $a_k=0$, $s_k \sim \param_0(s_{k-1},.)$; else $s_k$ is drawn from a
  Multinomial distribution with parameter $\param(\cdot,a_k)$.
\end{itemize}
The chains are initialized with $\PP(a_1 =1 \vert p_0) = 1- \PP(a_1=0\vert p_0) =p_0$; the
distribution of $s_1$ given $a_1=0$ and $\param$ (resp. given $a_1=1$ and $\param$) is uniform on
$\{1,\cdots, 4\}$ (resp. a Multinomial distribution with parameter $\param(\cdot, 1)$).

This description yields to the following conditional distribution of $A$ given
$\mathsf{S}$: (up to a multiplicative constant) - see
\cite{kou:zhou:wong:2006} for similar derivation -
\begin{align*}
  \mathrm{P}(A|S) &\propto  \frac{\Gamma(N_1(A)+b_1) \Gamma(N_0(A)+b_2)}{\Gamma(N_1(A)+N_0(A)+b_1+b_2)} \ \prod_{i=1}^w \frac{\prod_{j=1}^4 \Gamma(c_{j,i}(A)+\iota_{j,i})}{\Gamma(\sum_{\ell=1}^4 c_{\ell,i}(A) + \iota_{\ell,i})}  \cdots \\
  & \ \times \prod_{k=2}^L (\delta_{a_{k-1}+1}(a_k))^{\mathbf{1}_{a_{k-1} \in
      \{1, \dots,w-1\}}} \prod_{k=2}^L
  \left(\param_0(s_{k-1},s_k)\right)^{\un_{a_k=0}} \left(\un_{\{0\}}(a_1)
    \frac{1}{4} + \un_{\{1\}}(a_1) \right)
\end{align*}
where
\begin{itemize} 
\item[$\bullet$] $N_1(A)=\#\{k, a_k=1\}$ is the number of elements of $A$ equal to $1$.
 \item[$\bullet$] $N_0(A)=\#\{k, a_k=0\}$ is the number of elements of $A$ equal to $0$.
 \item[$\bullet$] $c_{j,i}(A)=\sum_{k=1}^L \mathbf{1}_{a_k=i} \mathbf{1}_{s_k=j}$ is the
   number of pairs $(a_k,s_k)$ equal to $(i,j)$.
\end{itemize}

\begin{figure}[ht!]
\begin{center}
  \includegraphics[height=6cm,width=12cm]{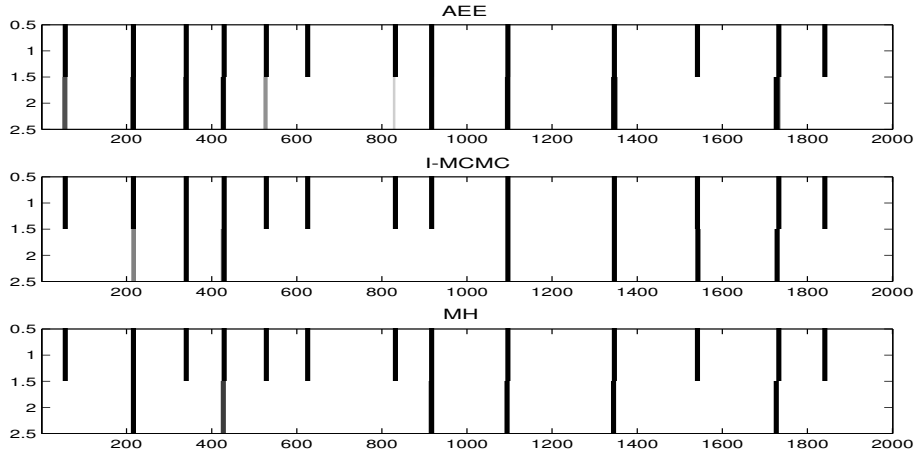}
\end{center}
\caption{Results given by  AEE, I-MCMC and a MH sampler}
\label{comparaison_bio}
\end{figure}

To highlight the major role of the equi-energy jumps, and the
  importance of the construction of the rings to make the acceptance
  probability of the jumps large enough, we compare AEE to I-MCMC, and to MH.
  The data are obtained with values of $p_0, \param_0$ and $\param$ similar to
  those of~\cite{kou:zhou:wong:2006}: $p_0=0.005$, $b_1 = 2 $, $b_2 = 200$, $\iota_{j,i}=1$ for all $j$, $i$, and
\begin{displaymath}
\param_0 = \left( \begin{array}{cccc}
0.1 & 0.7 & 0.1 & 0.1\\  
0.1 & 0.1 & 0.7 & 0.1\\
0.1 & 0.1 & 0.1 & 0.7\\
0.7 & 0.1 & 0.1 & 0.1\end{array} \right) , \,
 \param = \left( \begin{array}{cccccccccccc}
0.5 & 0.6 & 0.2 & 0.4 & 0.1 & 0.3 & 0.6 & 0.1 & 0.4 & 0.4 & 0.3 & 0 \\
0  & 0.2 & 0 & 0.2 & 0.8 & 0.7 & 0 & 0.9 & 0 & 0 & 0.2 & 0.3\\
0  & 0 & 0.8 & 0 & 0 & 0 & 0 & 0 & 0.3 & 0.5 & 0.4 & 0.1\\
0.5 & 0.2 & 0 & 0.4 & 0.1 & 0 & 0.4 & 0 & 0.3 & 0.1 & 0.1 & 0.6  \end{array} \right) .
\end{displaymath}
We sample a sequence $\mathsf{S}$ of length $L=2000$ and the size of the motif
is $w= 12$.

We now detail how the MH and the Metropolis-Hastings steps of AEE and I-MCMC
are run.  For the Metropolis-Hastings stage, the proposal distribution $p(A_n,
\tilde A_{n+1})$ is of the form 
\[
p(A_n, \tilde A_{n+1}) = q_0(\tilde a_1^{n+1}) \ \prod_{j=1}^{L-1} q_j(\tilde
a_{j}^{n+1},\tilde a_{j+1}^{n+1}; A_n) \eqsp,
\]
where we set $ \tilde A_{n+1} = (\tilde a_1^{n+1}, \cdots, \tilde a_L^{n+1})$.
The proposed state $\tilde{A}_{n+1}$ of the Metropolis-Hastings step is then
sampled element by element; the distributions are designed to be close to the
previous model: $\tilde{a}_{j+1}^{n+1}$ equal to $\tilde{a}_{j}^{n+1}+1$ if
$\tilde{a}_{j}^{n+1} \in \{1,\dots,w-1\}$, and else, $\tilde{a}_{j+1}^{n+1}$ is
sampled under a Bernoulli distribution of parameter 
\begin{equation}
\label{eq:proposal:bio}
  \frac{\hat{p}_0 \prod_{i=1}^{w} \hat{\param}_{A_n}(s_{j+i-1},i)}{\hat{p}_0
    \prod_{i=1}^{w} \hat{\param}_{A_n}(s_{j+i-1},i)+(1-\hat{p}_0)
    \prod_{i=1}^{w-1} \param_0(s_{j+i},s_{j+i+1})} \eqsp;
 \end{equation} the replacement constant $\hat{p}_0$ is fixed by the users
 and $\hat \param_{A_n}$ is given by $\hat{\param}_{A_n}(s,i) \propto
 c_{s,i}(A_n)+c$ - where $c$ is a value fixed by the users.  $q_0(\tilde
 a_1^{n+1})$ is the Bernoulli distribution with parameter
 (\ref{eq:proposal:bio}). Finally, the candidate $\tilde{A}_{n+1}$ is accepted
 with probability
\[
1 \wedge \frac{\mathrm{P}(\tilde A_{n+1}|S)^{1/T_k}}{\mathrm{P}(A_{n}|S)^{1/T_k}} \frac{p(\tilde{A}_{n+1}, A_n
  )}{p(A_n, \tilde{A}_{n+1} ) } \eqsp.
\]

Figure~\ref{comparaison_bio} displays the results obtained by AEE, I-MCMC and a
MH sampler.   Each subplot displays two horizontal lines with
  length equal to the length of the observed DNA sequence. The upper line
  represents the actual localization of the motifs, and the lower line
  represents in gray-scale the probability for each position to be part of a
  motif computed by one run of each algorithm after $2000$ iterations. For AEE
and I-MCMC, we choose $\InterractionProba=0.1$, $K=5$, $S=3$.  The acceptance
rate of the jump for AEE was about five times higher than for I-MCMC, which
confirms the interest of the rings. As expected, AEE performs better than the
other algorithms: there were $13$ actual motifs, and AEE retrieved $10$ motifs,
whereas the I-MCMC and the MH retrieved respectively $7$ and $6$ motifs.

\section{Conclusion}

As illustrated by the numerical examples, the efficiency of EE
  depends upon the choice of the energy rings. The adaptation we proposed
  improves this efficiency since it makes the probability of accepting a jump
  more stable.  It is known that adaptation can destroy the convergence of the
  samplers: we proved that AEE converges under quite general conditions on the
  adapted bounds and these general conditions can be used to prove the
  convergence of AEE when applied with other adaptation
  strategies~\cite{schreck:etal:2012}. It is also the first convergence result
  for an interacting MCMC algorithm including a selection mechanism.  Our
  sketch of proof can be a basis for the proof of other interacting MCMC such
  as the SIMCMC algorithm of~\cite{brockwell:delmoral:doucet:2010}, the
  Non-Linear MCMC algorithms described in~\cite[Section
  3]{andrieu:jasra:doucet:delmoral:2011} or the PTEEM algorithm
  of~\cite{bargatti:grimaud:pommeret:2012}.


\appendix

\section{Results on the transition kernels $P_\theta^{(k)}$}
Define
\begin{equation}
  \label{eq:ThetaTilde}
 G_\theta(x) \eqdef \int g_\theta(x,z) \theta(\rmd z) \eqsp, \qquad  \tilde \theta(x,\rmd y) \eqdef \frac{g_\theta(x,y) \theta(\rmd y)}{G_\theta(x)} \eqsp.
\end{equation}

\subsection{Proof of Proposition~~\ref{drift:small:aees}}
\label{secproof:drift:small:aees}
The case $k=1$ is a consequence of E\ref{EE3} since $P_\T^{(1)} = P^{(1)}$ for
any $\T$ so that $\pi_\T^{(1)} \propto \pi^{\beta_1}$. We now consider the case
$k \in \{2, \cdots, K\}$: in the proof below, for ease of notations we will
write $P$, $P_\T$, $W$, $\lambda, b$ and $\pi_\T$ instead of $P^{(k)}$,
$P_\T^{(k)}$, $W_k$, $\lambda_k, b_k$ and $\pi_\T^{(k)}$.

\textit{(a)} Let $m \geq 1$ and $\theta \in \tetaset_m$.  By definition of
$g_\theta$ (see (\ref{eq:FonctionSelectionEE})) and of $\tetaset_m$
(see~(\ref{eq:ThetaM})), $1/m \leq \int g_{\theta}(x,y) \theta(\rmd y) \leq S$.
Moreover, by E\ref{EE3}\ref{EE3:b}
\[
P_{\theta}W(x) = (1-\InterractionProba) PW(x) + \InterractionProba K_{\theta}W(x) \leq
(1-\InterractionProba) (\lambda W(x) + b )+ \InterractionProba K_{\theta}W(x) \eqsp.
\]
We have by (\ref{eq:DefinitionKthetaAEE}), (\ref{eq:DefinitionW}) and
(\ref{eq:ThetaTilde})
\begin{align*}
  K_\theta W(x) &=W(x)+ \int W(y) \alpha_{\theta}(x,y) \left(
    1-\frac{\pi^{\tau_k \beta_k}(y)}{\pi^{\tau_k \beta_k}(x)}\right) \tilde{\theta}
  (x,\rmd y) \eqsp.
\end{align*}
By (\ref{eq:DefinitionAlphaAEE}),
\[
K_\theta W(x) \leq W(x) + m \int_{\{y, \pi(y) \leq \pi(x) \}} W(y)
\frac{\pi^{\beta_{k} -\beta_{k-1}}(y)}{\pi^{\beta_{k} -\beta_{k-1}}(x)} \left(
  1-\frac{\pi^{\tau_k \beta_k}(y)}{\pi^{\tau_k \beta_k}(x)}\right)
g_\theta(x,y) {\theta} (\rmd y) \eqsp.
\]

Defining $\psi$ by $\psi(\sigma) = \sigma/(\sigma+1)^{(\sigma+1)/\sigma}$ gives
the upper bound $\sup_{z\in[0,1]}z(1-z^{\sigma}) \leq \psi(\sigma) $. Hence,
$K_\theta W(x) \leq W(x) + S m \ \psi\left( \tau_k \beta_k /(\beta_{k} -
  \beta_{k-1}) \right) {\theta}(W)$.  This yields $P_\theta W(x) \leq \tilde
\lambda W(x) + \tilde b m \theta(W)$ with $\tilde{\lambda} =
(1-\InterractionProba)\lambda+\InterractionProba<1$ and $\tilde b =
\InterractionProba S \ \psi \left(\tau_k \beta_k /(\beta_{k} - \beta_{k-1})
\right) + (1-\InterractionProba) b$.  The minorization condition comes from the
lower bound $ P_{\theta}(x,A) \geq (1-\InterractionProba) P(x,A)$.

\textit{(b)} Let $m \geq 1$ and $\theta \in \tetaset_m$.  By
E\ref{EE3}\ref{EE3:a}, $P$ is  $\varphi$-irreducible and so is $P_\theta$;
$P_\theta$ possesses a $1$-small set and is thus aperiodic. In addition, $
P_{\T}W \leq ( 1 + \tilde \lambda) W/ 2 + \tilde{b} \T(W) \un_{\{W
  \leq c \}}$, with $c\eqdef 2 \tilde{b} m \ \T(W) (1-\tilde{\lambda})^{-1}$
and $\{W \leq c \}$ is a $1$-small set for $P_\theta$. By \cite[Chapter
15]{meyn:tweedie:1993}, $\pi_\theta$ exists and $\pi_\theta(W) \leq \tilde b m
\theta(W) (1-\tilde \lambda)^{-1}$.

\subsection{Ergodic behavior}
\begin{lemma} \label{lemme:unif:ergo}
  Assume E\ref{EE1}\ref{EE1:a} and E\ref{EE3}.  Then for all $a \in (0,1)$, for
  all $m \geq 1$ and all $\theta \in \tetaset_{m}$, there exist $C_{\theta}$
  and $\rho_{\theta} \in (0,1)$ such that for all $x \in \mathbf{X}$ and any $j
  \geq 1$ and any $k \in \{1, \cdots, K\}$,
   \begin{align} \label{control:normev}
     \|\left(P_{\T}^{(k)}\right)^j(x,.)-\pi_{\T}^{(k)}\|_{W_{k}^a} \leq C_{\T} \ \rho_{\T}^j \ W^a_k(x) \eqsp.
   \end{align}
   Let $k \in \{1, \cdots, K-1 \}$ and assume in addition that $\lim_{n \to
     \infty} \theta_n^{(k)}(W_k) = \T_\star^{(k)}(W_k)$ w.p.1. Then for any
   positive integer $q$, on the set $\bigcap_{n \geq q} \{\theta_n \in
   \tetaset_{m_\star} \}$
   \begin{align}
     \limsup_n \rho_{\theta_n^{(k)}}<1, \limsup_n C_{\theta_n^{(k)}} < +\infty
     , \PP-\text{a.s.  } \eqsp. \label{lemme:unif:ergo:tool1}
   \end{align}
\end{lemma}
\begin{proof} The proof in the case $k=1$ is a
  consequence of E\ref{EE3} and \cite[Chapter 15]{meyn:tweedie:1993} since
    $P_\T^{(1)} = P^{(1)}$. Consider the case $k \geq 2$.  Here again, the
    dependence upon $k$ is omitted: $P_\T, W, \T_n$ denote $P_\T^{(k)}, W_k$
    and $\T_n^{(k)}$.
  
  {\em Proof of (\ref{control:normev})} Let $a \in (0,1)$ and set $V=W^a$. By
  the Jensen's inequality and Proposition~\ref{drift:small:aees}, there exists
  $\bar \lambda \in (0,1)$ and $\bar b$ such that for any $m\geq 1$ and any
  $\theta \in \tetaset_m$,
\[
P_\theta V \leq \bar \lambda V + \bar b  \, m \,  \theta(W)^a \eqsp.
\]
Let $m \geq 1$ and $\theta \in \tetaset_m$. By \cite[Lemma
2.3.]{fort:moulines:priouret:2010}, (\ref{control:normev}) holds and there
exist constants $C, \gamma >0$ such that for any $\theta \in \tetaset_m$,
\[
C_\theta \vee (1-\rho_\theta)^{-1} \leq C \left( \bar{b} \ m \theta(W) \vee
  \delta_\theta^{-1} \vee (1-\bar{\lambda})^{-1} \right)^\gamma \eqsp,
\]
where $\delta_\theta$ is the minorizing constant of $P_\theta$ on the set $\{x:
W(x) \leq 2 \bar b m \, \theta(W) \, (1-\bar \lambda)^{-1} -1 \}$. 

{\em Proof of (\ref{lemme:unif:ergo:tool1})} For all $\omega \in \bigcap_{n
  \geq q} \{\T_n \in \tetaset_{m_\star} \}$,  
\[
\limsup_n \{ C_{\theta_n(\omega)} \vee (1-\rho_{\theta_n(\omega)})^{-1} \} \leq
C \left( \bar{b} \, m \, \limsup_n \theta_n(W) \vee \limsup_n
  \delta_{\theta_n(\omega)}^{-1} \vee (1-\bar{\lambda})^{-1} \right)^\gamma
\eqsp.
\]
Since $\limsup_n \theta_n(W) = \theta_\star(W) < \infty$ w.p.1, $\limsup_n
\delta_{\theta_n(\omega)}^{-1} < \infty$ w.p.1. thus showing that on the set
$\bigcap_{n \geq q} \{\T_n \in \tetaset_{m_\star} \}$, $\limsup_n \{
C_{\theta_n(\omega)} \vee (1-\rho_{\theta_n(\omega)})^{-1} \} < \infty$. This
implies (\ref{lemme:unif:ergo:tool1}).
\end{proof}

\subsection{Moment conditions}
\label{app:lemme:unif:moments}
Let $m_\star >0$. Define for any positive integer $q$ and any $ k \in \{1,
\cdots, K-1\}$,
\[
A_{q,n}^{(k)} = \bigcap_{\ell\leq k} \bigcap_{q \leq j\leq n} \left\{
  \theta_j^{(\ell)} \in \Theta_{m_\star} \right\} \quad \text{if $q \leq n$,}
\quad \textrm{and} \quad A_{q,n}^{(k)}= \Omega \quad \text{otherwise;}
\]
by convention, $A_{q,n}^{(0)} = \Omega$ for any $q,n \geq 0$.

\begin{lemma} \label{lemme:unif:moments}
  Assume E\ref{EE1}\ref{EE1:a}, E\ref{EE3} and
  $\PE\left[W_k(Y_0^{(k)})\right]< \infty$ for any $ k\in \{1, \cdots, K\}$.
  Then for any $k \in \{1, \cdots, K \}$,
\begin{equation}
  \label{lemme:unif:ergo:tool2}
 \sup_{j\geq 1}  \PE\left[W_k(Y_j^{(k)})  \un_{A_{q,j-1}^{(k-1)}}  \right] < \infty \eqsp.
 \end{equation}

\end{lemma}
\begin{proof} The proof is by induction on $k$. The case $k=1$ is a consequence of E\ref{EE3} since $P_\T^{(1)} =
  P^{(1)}$. Assume the property holds for $k \in \{2, \cdots, K-1\}$. In this
  proof, $W_{k+1}, P_\T^{(k+1)}, \T_n^{(k)}, Y^{(k)}, Y^{(k+1)}, P^{(k+1)}$,
  $K_\T^{(k+1)}$ will be denoted by $W, P_\T, \T_n, Y, X, P, K_\T$.
  
  By (\ref{eq:interactingMCMC}) and Proposition~\ref{drift:small:aees} we
  obtain, for $j > q$
\begin{align*}
  \PE\left[ W(X_j) \un_{A_{q,j-1}^{(k)}} \right] & \leq \PE\left[
    P_{\theta_{j-1}} W(X_{j-1}) \un_{A_{q,j-1}^{(k)}}\right]
  \\
  &\leq \tilde{\lambda} \PE\left[ W(X_{j-1}) \un_{A_{q,j-2}^{(k)} }\right]
  +\tilde{b} \, m_\star
  \, \PE\left[\theta_{j-1}(W)  \un_{A_{q,j-1}^{(k-1)} }  \right] \\
  & \leq \tilde{\lambda} \PE\left[ W(X_{j-1}) \un_{A_{q,j-2}^{(k)} }\right]
  +\tilde{b} \, m_\star \, \sup_l \PE\left[W(Y_l) \un_{A_{q,l-1}^{(k-1)} }
  \right] \eqsp.
\end{align*}
Since $W_{k+1} \in \mathcal{L}_{W_k}$, the induction assumption implies that $
\sup_l \PE\left[W(Y_l) \un_{A_{q,l-1}^{(k-1)} } \right]< \infty$. Iterating
this inequality allows to write that for some constant $C'$
\[
\sup_{j \geq q} \PE\left[ W(X_j) \un_{A_{q,j-1}^{(k)}} \right] \leq C' \ 
\PE\left[ W(X_q) \right] \eqsp.
\]
Finally, by definition of $P_{\T_j}$, either $P_{\T_j} = P$ if $ \T_j \notin
\bigcup_m \Theta_m$, or $P_{\T_j} = (1-\epsilon) P + \epsilon K_{\T_j}$
otherwise; note that if $\T_j \in \bigcup_m \Theta_m$ then $\T_j \in
\Theta_{1/j}$. Since both $P$ and $P_\T$ for $\T \in \bigcup_m \Theta_m$
satisfy a drift inequality (see E\ref{EE3} and Proposition
\ref{drift:small:aees}), $\PE\left[ W(X_q) \right] < \infty$ by
(\ref{eq:interactingMCMC}).
\end{proof}

\section{Proof of Theorem~\ref{theo:tout}} \label{proof}
  \begin{hypA}\hspace{-0.1cm}$(k)$
    \label{hyp:proof:subset} There exists $m_\star >0$ such that $\PP\left(\bigcup_{q \geq 1} \bigcap_{n \geq q} \{ \T_n^{(k)} \in \Theta_{m_\star} \} \right)=1$.
  \end{hypA}
 \begin{hypA}\hspace{-0.1cm}$(k)$
    \label{hyp:proof:pi} for any $a \in (0,1)$ and any continuous function $f \in \mathcal{L}_{W_k^a}$,
\[
\lim_{n \to \infty} \pi_{\T_n^{(k-1)}}^{(k)}(f) = \T_\star^{(k)}(f) \eqsp.
\]
  \end{hypA}
\begin{hypA}\hspace{-0.1cm}$(k)$\label{hyp:proof:ergo} For all bounded continuous function $f$, $\lim_{n \to \infty} \PE\left[ f(Y_n^{(k)})\right] = \T_\star^{(1)}(f)$.
 \end{hypA}
\begin{hypA}\hspace{-0.1cm}$(k)$ \label{hyp:proof:lln} $\theta_{\star}^{(k)}(W_{k+1}) < +\infty$, and for any $a \in (0, \frac{1 + \Gamma}{2} \wedge 1)$ and any
  continuous function $f$ in $\mathcal{L}_{W_k^a}$, $\theta_n^{(k)}(f)
  \rightarrow \theta_{\star}^{(k)}(f)$ a.s.
\end{hypA}

By Proposition~\ref{prop:lgn:first}, the conditions R\ref{hyp:proof:ergo} and
R\ref{hyp:proof:lln} hold for $k=1$; R\ref{hyp:proof:pi} also holds for $k=1$
since $\pi_\T^{(1)} = \T_\star^{(1)}$ for any $\T$.  We assume that for any $j
\leq k$, for $k \in \{1, \cdots, K-1 \}$, the conditions
R\ref{hyp:proof:subset}($j-1$), R\ref{hyp:proof:pi}($j$),
R\ref{hyp:proof:ergo}($j$) and R\ref{hyp:proof:lln}($j$) hold. We prove that
R\ref{hyp:proof:subset}($k$), R\ref{hyp:proof:pi}($k+1$),
R\ref{hyp:proof:ergo}($k+1$) and R\ref{hyp:proof:lln}($k+1$) hold. To make the
notations easier, the superscript $k$ is dropped from the notations: the
auxiliary process $Y^{(k)}$ will be denoted by $Y$, and the process $Y^{(k+1)}$
by $X$; $P^{(k+1)}, W_{k+1}, K_\T^{(k+1)}, P_\T^{(k+1)}, \alpha_\T^{(k+1)},
\pi_{\T}^{(k+1)}$ and $\T_n^{(k)},\T_\star^{(k)}$ are resp. denoted by $P, W,
K_\T, P_\T,\alpha_\T, \pi_{\T}$ and $ \T_n, \T_\star$.

Finally, we define the V-variation of the two kernels $P_\theta$ and $P_{\theta'}$
by:
\begin{displaymath}
  D_V(\theta,\theta')=\sup_{x\in \mathbf{X}} \left( \frac{\|P_{\theta}(x,.)-P_{\theta'}(x,.) \|_V}{V(x)}\right) \eqsp.
\end{displaymath}
When $V=1$, we will simply write $D$.

\subsection{Proof of R\ref{hyp:proof:subset}($k$)}
\label{secproof:prop:subsetTheta}
The proof is prefaced with a preliminary lemma.
\begin{lemma} \label{lemme:conv:h} 
    For all $l \in \{1, \cdots, S-1 \}$ and any $\theta, \theta'$,
\[
\sup_{x \in \mathbf{X}} \left|h_{\theta,l}(x) - h_{\theta',l}(x) \right| \leq
\frac{1}{r} \sup_{l \in \{1,\cdots, S-1 \}}\left| \bornea_{\theta,l} -
  \bornea_{\theta',l} \right| \eqsp.
\]
\end{lemma}
\begin{proof}
  Note that $|(1-a)_+ - (1-b)_+| \leq |b-a|$. Therefore, for all $x \in \X$ :
\[
|h_{\theta,l}(x) - h_{\theta',l}(x)| \leq \frac{\left| d(\mesobs(x),
    \anneau_{\theta,l}) -d(\mesobs(x),\anneau_{\theta',l}) \right|}{r} \eqsp.
\]
This concludes the proof.
\end{proof}
{\em (Proof of R\ref{hyp:proof:subset}($k$))} We prove there exist an integer
$m_\star \geq 1$ and a positive r.v. $N$ such that
\[
\PP(N < \infty)=1 \eqsp, \qquad \PP\left( \bigcap_{n \geq N} \left \{ \inf_x
    \int g_{\theta_n}(x,y) \theta_n(\rmd y) \geq 1/m_\star \right\}\right)=1 \eqsp.
\]
To that goal, we prove that with probability $1$, for all $n$ large enough,
\begin{align}
  \inf_x \int g_{\theta_n}(x,y) \theta_n(\rmd y) \geq \inf_{\ell \in \{1,
    \cdots, S-1 \}} \int h_{\theta_\star, \ell}(y) \ \theta_\star(\rmd y)
  \eqsp,
  \label{eq:tool2:prop:subsetTheta}
\end{align}
and use the assumption E\ref{EE:ring}\ref{ring:c}. For all $x$ and $\T$, there
exists a ring index $l_{x,\theta} \in \{1, \cdots, S \}$ such that $\mesobs(x)
\in \anneau_{\theta,l_{x,\theta}}$. Upon noting that $d(\mesobs(x),
\anneau_{\theta,\ell_{x,\theta}}) =0$, it holds
\begin{align*}
  \liminf_n \inf_x \int g_{\theta_n}(x,y) \theta_n(\rmd y) \geq \liminf_n \inf_{l \in \{1,
    \cdots, S\}} \int h_{\theta_n, l}(y) \theta_n(\rmd y) \eqsp.
\end{align*}
We write
\begin{align*}
  \int h_{\theta_n, l}(y) \theta_n(\rmd y) & \geq \int h_{\theta_\star, l}(y)
  \theta_n(\rmd y)
  - \int  \left| h_{\theta_n, l}(y) - h_{\theta_\star, l}(y) \right| \theta_n(\rmd y) \\
  & \geq \int h_{\theta_\star, l}(y) \theta_n(\rmd y) - \sup_{y \in \X} \left|
    h_{\theta_n, l}(y) - h_{\theta_\star, l}(y) \right| \eqsp.
\end{align*}
By definition of $h_{\theta_\star,\ell}$, $y \mapsto h_{\theta_\star, l}(y)$ is
continuous and bounded. Therefore, by R\ref{hyp:proof:lln}(k),
Lemma~\ref{lemme:conv:h} and E\ref{EE:ring}\ref{ring:a}, the proof of
(\ref{eq:tool2:prop:subsetTheta}) is concluded by
\[
\liminf_n \int h_{\theta_n, l}(y) \theta_n(\rmd y) > \int h_{\theta_\star, l}(y)
\theta_\star(\rmd y) \eqsp.
\]

\subsection{Proof of  R\ref{hyp:proof:pi}($k+1$)}
\label{secproof:prop:cvg:pi}
First of all, observe that by definition of $\pi_\T$ (see
Proposition~\ref{drift:small:aees}) and the expression of $P_\T$,
$\pi_{\T_\star} \propto \pi^{\beta_{k+1}}$. We check the conditions of
\cite[Theorem 2.11]{fort:moulines:priouret:2010}.  By
Proposition~\ref{prop:subsetTheta} it is sufficient to prove that for any  $q
\geq 1$, $\lim_{n \to \infty} |\pi_{\theta_n}(f) - \pi_{\theta_\star}(f) | \un_{\bigcap_{j
    \geq q} \{\theta_j \in \tetaset_{m_\star} \}} =0$ w.p.1.

\paragraph{Case $f$ bounded}   
Lemma~\ref{lemme:unif:ergo} and R\ref{hyp:proof:lln}($k$) show that on the set
$\bigcap_{j \geq q} \{\theta_j \in \tetaset_{m_\star}\}$, $\limsup_n
C_{\theta_n} < \infty$ and $\limsup_n (1-\rho_{\theta_n})^{-1} < \infty$ w.p.1.
Equicontinuity of the class $\{P_\theta f, \theta \in \tetaset_{m_\star} \}$,
where $f$ is a bounded continuous function on $\X$, will follow from
Lemmas~\ref{lem:ReguX:G} to \ref{lem:Equicont}. Finally, the weak convergence
of the transition kernels is proved in Lemma~\ref{lem:AScvgNoyaux}.

\paragraph{Case $f$ unbounded} Following the same lines as in
the proof of \cite[Theorem 3.5]{fort:moulines:priouret:2010}, it can be proved
that the above discussion for $f$ bounded and
Proposition~\ref{drift:small:aees}\eqref{drift:small:aees:item2} imply
\begin{align*}
 \lim_{n \to \infty}
\{\pi_{\theta_n}(f) - \pi_{\T_\star}(f) \}\un_{\cap_{j \geq q} \{\theta_j \in
  \Theta_{m_\star} \}} =0
\end{align*}
w.p.1.  for any continuous function $f$ such that
$|f|_{W^a_{k+1}}< \infty$.

\begin{lemma}
\label{lem:ReguX:G}
For all $\theta \in \bigcup_m \tetaset_m$, and $ x,x'$, $\sup_y \left| g_\theta(x,y) - g_\theta(x',y) \right| \leq \frac{S}{r} | \pi(x) - \pi(x')|$.
\end{lemma}
\begin{proof}
  By (\ref{eq:FonctionSelectionEE}),
  \begin{align*}
    |g_{\theta}(x,y)-g_{\theta}(x',y)| \leq \sum_{l=1}^S
    |h_{\theta,l}(x)-h_{\theta,l}(x')| h_{\theta,l}(y) \leq \sum_{l=1}^S
    |h_{\theta,l}(x)-h_{\theta,l}(x')| \eqsp.
\end{align*}
The proof is completed since
\[
|h_{\theta,l}(x)-h_{\theta,l}(x')| \leq
\frac{|d(\pi(x),\anneau_{\theta,l})-d(\pi(x'),\anneau_{\theta,l})|}{r} \leq
\frac{|\pi(x)-\pi(x')|}{r} \eqsp.
\]
\end{proof}
\begin{lemma}
   \label{lem:ReguX:alpha} Assume E\ref{EE1}\ref{EE1:a}.
   For all $m \geq 1$, there exists a constant $C_m$ such that for all
   $x,x',y,y' \in \X$ and $\T \in \Theta_m$ 
\begin{align} \label{majoration:alpha:x}
  \left| \alpha_\theta(x,y) - \alpha_{\theta}(x',y) \right| &\leq C_m \, \left[
    \left| \pi^{\beta_k-\beta_{k+1}}(x)-\pi^{\beta_k-\beta_{k+1}}(x') \right|
    + \left| \pi(x) - \pi(x')\right| \right] \eqsp, \\
 \label{majoration:alpha:y}
 \left| \alpha_\theta(x,y) - \alpha_{\theta}(x,y') \right| & \leq C_m
 \left[\left| \pi^{\beta_k-\beta_{k+1}}(y)-\pi^{\beta_k-\beta_{k+1}}(y')
   \right| + \left| \pi(y) - \pi(y')\right| \right] \eqsp.
\end{align}
\end{lemma}
\begin{proof}
  By definition of $\alpha_\theta$ (see (\ref{eq:DefinitionAlphaAEE})),
  $\alpha_{\theta}(x,y)-\alpha_{\theta}(x',y) = \left( 1 \wedge a \right) -
  \left( 1 \wedge b \right) $, with
\begin{align*}
  a=\frac{\pi^{\beta_{k+1}-\beta_{k}}(y)}{\pi^{\beta_{k+1}-\beta_{k}}(x)}
  \frac{\int g_{\theta}(x,z) \theta(\rmd z)}{\int g_{\theta}(y,z) \theta(\rmd z) }
  \quad \textrm{and} \quad
  b=\frac{\pi^{\beta_{k+1}-\beta_{k}}(y)}{\pi^{\beta_{k+1}-\beta_{k}}(x')}
  \frac{\int g_{\theta}(x',z) \theta(\rmd z)}{\int g_{\theta}(y,z) \theta(\rmd z) }
  \eqsp.
\end{align*}
Note that $|(1 \wedge a) - (1 \wedge b)| \leq |a-b| \left( \un_{a \leq 1} +
  \un_{ b \leq 1, a > 1 } \right)$. By symmetry, we can assume that $b \leq 1$
and this implies 
\begin{align*}
  \frac{\pi^{\beta_{k+1}-\beta_{k}}(y)}{\pi^{\beta_{k+1}-\beta_{k}}(x')} \leq \frac{\int g_{\theta}(y,z)
    \theta(\rmd z) }{\int g_{\theta}(x',z) \theta(\rmd z)} \leq S m \eqsp,
\end{align*}
since $g_\T(x,y) \leq S$. Therefore,
\begin{multline*}
  |a-b| = \frac{\pi^{\beta_{k+1}-\beta_{k}}(y)}{\int  g_{\theta}(y,z) \theta(\rmd z) } \left|\frac{\int g_{\theta}(x,z) \theta(\rmd z)}{\pi^{\beta_{k+1}-\beta_{k}}(x)} -\frac{\int g_{\theta}(x',z) \theta(\rmd z)}{\pi^{\beta_{k+1}-\beta_{k}}(x')}  \right| \\
  \leq S m \left[ \pi^{\beta_{k+1}-\beta_{k}}(y) \left|
      \pi^{\beta_k - \beta_{k+1}}(x)-\pi^{\beta_k -\beta_{k+1}}(x') \right| + m
    \left|\int (g_{\theta}(x,z) - g_{\theta}(x',z)) \theta(\rmd z) \right| \right]
  \eqsp.
\end{multline*}
The proof of (\ref{majoration:alpha:x}) is concluded by
Lemma~\ref{lem:ReguX:G}. The proof of (\ref{majoration:alpha:y}) is on the same
lines and omitted.
\end{proof}

\begin{lemma}
  \label{lem:Equicont}
  Assume E\ref{EE1} and E\ref{EE3}\ref{EE3:a}. For any $m \geq 1$ and for any
  continuous bounded function $f$, the class of functions $\{P_\theta f, \theta
  \in \tetaset_m \}$ is equicontinuous.
\end{lemma}
\begin{proof}
  Let $f$ be a continuous function on $\X$, bounded by $1$. Let $m \geq 1$ and
  $\theta \in \tetaset_m$.  We have
\begin{align*}
  P_\theta f(x) - P_{\theta} f(x') = & (1-\InterractionProba) \left( Pf(x) - Pf(x')
  \right) + \InterractionProba \left(f(x) - f(x') \right) \left(1-\int
    \alpha_\theta(x',y)
    \tilde \theta(x,\rmd y)\right) \\
  &  + \InterractionProba \int \left( f(y) -f(x) \right) \left( \alpha_\theta(x,y) - \alpha_{\theta}(x',y) \right) \tilde {\theta}(x,\rmd y) \\
  & + \InterractionProba \int \alpha_{\theta}(x',y) (f(y)-f(x'))
  (\tilde{\theta}(x,\rmd y)-\tilde{\theta}(x',\rmd y)) \eqsp,
\end{align*}
where $\tilde \theta$ is given by (\ref{eq:ThetaTilde}).  This yields to
\begin{align*}
  \left| P_\theta f(x) - P_{\theta} f(x') \right| \leq & \left| Pf(x) - Pf(x')
  \right| +  \left|f(x) -f(x') \right| \\
  & + 2 \sup_y \left| \alpha_\theta(x,y)- \alpha_{\theta}(x',y)\right| +2
  \left\| \tilde \theta(x,.) - \tilde {\theta}(x',.) \right\|_{\mathrm{TV}} \eqsp.
\end{align*}
We have
\begin{align*}
 \|\tilde {\theta}(x,.) - \tilde {\theta}(x',.) \|_{\mathrm{TV}} & \leq
    \frac{1}{G_\theta(x)} \sup_y \left| g_\theta(x,y) - g_\theta(x',y) \right|
    +
    \frac{S}{G_\theta(x) G_\theta(x')} \left| G_\theta(x) - G_\theta(x') \right| \\
 &  \leq m \sup_y \left| g_\theta(x,y) - g_\theta(x',y) \right| +
    S m^2 \sup_y \left| g_\theta(x,y) - g_\theta(x',y) \right| \eqsp,
\end{align*}
where $G_\theta$ is given by (\ref{eq:ThetaTilde}).  So
Lemmas~\ref{lem:ReguX:G} and \ref{lem:ReguX:alpha} imply that for all $m \geq
1$, there exists a constant $C_m$ such that for all $\T \in \tetaset_m$:
\begin{align} \label{difference:pteta:x}
  \left| P_\theta f(x) - P_{\theta} f(x') \right| \leq & \left| Pf(x) - Pf(x')
  \right| +  \left|f(x) -f(x') \right| \nonumber\\
  & + C_m \left( | \pi(x) - \pi(x')|+ |\pi^{\beta_k-\beta_{k+1}}(x) - \pi^{\beta_k-\beta_{k+1}}(x')
    |\right) \eqsp.
\end{align}
The proof is concluded since $P$ is Feller and $\pi$ is continuous.
\end{proof}

\begin{lemma}
 \label{lem:AScvgNoyaux}
 Let $m \geq 1$. Assume E\ref{EE1}, E\ref{EE:ring}\ref{ring:a} and
 R\ref{hyp:proof:lln}(k). For all $x \in \X$, there exists a set $\Omega_{x}$
 such that $\PP(\Omega_{x})=1$ and for all $\omega \in \Omega_{x}$ and any
 bounded continuous function $f$ 
\[
\lim_{n \to \infty} \left| P_{\T_n(\omega)}f(x) - P_{\T_\star}f(x) \right| \un_{\bigcap_j
  \{\theta_j \in \tetaset_m \}}= 0 \eqsp.
\]
\end{lemma}
\begin{proof}
  Following the same lines as in the proof of \cite[Proposition
  3.3.]{fort:moulines:priouret:2010}, it is sufficient to prove that for any $x
  \in \X$ and any bounded continuous function $f$, $\lim_{n \to \infty} P_{\theta_n}(f) =
  P_{\theta_\star}(f)$ w.p.1 on the set $\bigcap_j \{\theta_j \in \tetaset_m
  \}$.  Let $f$ and $x$ be fixed.  We write
\begin{align} \label{diff:P}
  P_\theta f(x) - P_{\theta'} f(x) = & \InterractionProba \int \left(
    \alpha_\theta(x,y) -
    \alpha_{\theta'}(x,y) \right) \left(f(y) -f(x) \right) \tilde \theta(x, \rmd y) \nonumber \\
  & + \InterractionProba \int \alpha_{\theta'}(x,y) \left(f(y) -f(x) \right)
  \left( \tilde{\T}(x,\rmd y) - \tilde{\T'}(x,\rmd y) \right) \eqsp,
\end{align}
where $\tilde \theta$ is given by (\ref{eq:ThetaTilde}).
Moreover,
\begin{align*}
  \tilde{\T}(x,\rmd y) - \tilde{\T'}(x,\rmd y)
& =\frac{g_\theta(x,y) \theta(\rmd y) -g_{\theta'}(x,y) \theta'(\rmd y)}{G_\T(x)} +
g_{\theta'}(x,y) \theta'(\rmd y)  \, \frac{(G_{\T'}(x) - G_\T(x))}{G_\T(x)G_{\T'}(x) } \eqsp.
\end{align*}
This yields to
  \begin{multline*}
    \InterractionProba^{-1} \left( P_{\theta_n}f(x) - P_{\theta_\star} f(x)
    \right) = \int \left( \alpha_{\theta_n}(x,y) - \alpha_{\theta_\star}(x,y)
    \right)\ \left( f(y) - f(x) \right) \tilde
    \theta_n(x, \rmd y)  \\
    - \int \frac{g_{\theta_{\star}}(x,y)}{G_{\theta_n}(x)}F(x,y)\left( \theta_\star( \rmd y) - \theta_n(\rmd y) \right) \\
- \int F(x,y) \left( (g_{\T_n}(x,y)-g_{\T_{\star}}(x,y)) \frac{\T_n(\rmd y)}{G_{\T_n}(x)} + g_{\theta_{\star}}(x,y) \theta_{\star}(\rmd y)  \, \frac{(G_{\T_{\star}}(x) - G_{\T_n}(x))}{G_{\T_n}(x)G_{\T_{\star}}(x) } \right)
    \eqsp,
  \end{multline*}
  where $F(x, y) = \alpha_{\theta_\star}(x,y) \left( f(y) - f(x) \right)$.
  There exists a constant $C_m$ such that on the set $\bigcap_n \{\theta_n \in
  \tetaset_m\}$, (see the proof of Lemma~\ref{lem:ReguX:alpha} for similar
  upper bounds)
  \begin{multline*}
     \left| \alpha_{\theta_n}(x,y) -
      \alpha_{\theta_\star}(x,y)\right| \leq C_m \, \left|
      \frac{G_{\theta_n}(x)}{G_{\theta_n}(y)} -
      \frac{G_{\theta_\star}(x)}{G_{\theta_\star}(y)}\right|  \\
    \leq m^2 S C_m \, \left(\left| G_{\theta_n}(x) - G_{\theta_\star}(x)
      \right|+\left| G_{\theta_n}(y) - G_{\theta_\star}(y) \right| \right)
  \end{multline*}
  where $G_\theta(x)$ is defined by (\ref{eq:ThetaTilde}). We write by
  definition of the function $g_\theta$ (see (\ref{eq:FonctionSelectionEE}))
  \begin{multline*}
    \sup_{x} \left| G_{\theta_n}(x) - G_{\theta_\star}(x) \right| \leq
    \sup_{x,z} |g_{\theta_n}(x,z) - g_{\theta_\star}(x,z)| + \sup_{x} \left| \int
      g_{\theta_\star}(x,z) \theta_n(\rmd z) - \int g_{\theta_\star}(x,z)
      \theta_\star(\rmd z) \right|
    \\
    \leq 2 \sum_{l=1}^S \sup_z |h_{\theta_n,l}(z) - h_{\theta_\star,l}(z)| +
    \sum_{l=1}^S \left| \int h_{\theta_\star,l}(z) \theta_n(\rmd z) - \int
      h_{\theta_\star,l}(z) \theta_\star(\rmd z) \right| \eqsp.
  \end{multline*}
  By Lemma~\ref{lemme:conv:h} and E\ref{EE:ring}\ref{ring:a}, the first term
  converges to zero w.p.1. Since $t \mapsto h_{\theta_\star,l}(t)$ is
  continuous, R\ref{hyp:proof:lln}(k) implies that the second term tends to
  zero w.p.1. Therefore, on the set $\bigcap_n \{\theta_n \in \tetaset_m\}$,
  $\sup_{x,y} \left| \alpha_{\theta_n}(x,y) -
    \alpha_{\theta_\star}(x,y)\right|$ converges to zero w.p.1, as well as
  $\sup_{x,y} |g_{\theta_n}(x,y) - g_{\theta_\star}(x,y)|$, and $ \sup_{x}
  \left| G_{\theta_n}(x) - G_{\theta_\star}(x) \right| $ .

  Note that by Lemma~\ref{lem:ReguX:alpha}, $y \mapsto F(x,y)$ is bounded and
  continuous. Therefore, following the same lines as above, it can be proved
  that under R\ref{hyp:proof:lln}(k) and E\ref{EE:ring}\ref{ring:a}, on the set
  $\bigcap_n \{\theta_n \in \tetaset_m\}$, $\lim_{n \to \infty} |\int F(x,y)\theta_n(x, \rmd
  y) -\int F(x,y) \theta_\star(x, \rmd y)|$ =0 w.p.1
\end{proof}

\subsection{Proof of R\ref{hyp:proof:ergo}(k+1)}
\label{secproof:ergo_ees}
We check the conditions of \cite[Theorem 2.1]{fort:moulines:priouret:2010}.
Let $f$ be a bounded continuous function on $\X$. By R\ref{hyp:proof:pi}(k+1),
$\lim_{n \to \infty} \pi_{\theta_n}(f) = \pi_{\T_\star}(f) \propto \pi^{\beta_{k+1}}$
w.p.1. Let $\delta >0$. By Proposition~\ref{prop:subsetTheta}, there exists $q
\geq 1$ such that $\PP( \bigcap_{n \geq q} \{\theta_n \in \tetaset_{m_\star}\}
) \geq 1- \delta$.  Following the same lines as in the proof of \cite[Theorem
3.4]{fort:moulines:priouret:2010}, it can be proved by using
Lemmas~\ref{lemme:unif:ergo}, \ref{lemme:conv:h} and \ref{lemme:cv:dv:proba}
and the condition E\ref{EE:ring}\ref{ring:b} that $\lim_{n \to \infty} \PE\left[
  \left(f(X_n) - \pi_{\theta_n}(f) \right) \un_{\bigcap_{n \geq q} \{\theta_n
    \in \tetaset_{m_\star}\}} \right] =0$.  This concludes the proof.
\begin{lemma} \label{lemme:cv:dv:proba}
  For all $m \geq 1$, there exists a constant $C_m$ such that for any $\theta,
  \theta' \in \tetaset_m$,
\[
D(\theta, \theta') \leq C_m \left( \| \theta - \theta' \|_{\mathrm{TV}} +
  \sup_{l, x} \left| h_{\theta,l}(x) -h_{\theta',l}(x) \right|\right) \eqsp.
\]
\end{lemma}
\begin{proof}
  By definition of $P_\T$, for all function $f$ bounded by $1$, (\ref{diff:P})
  holds.  So
\begin{align*}
  D(\theta,\theta') &\leq 2 \InterractionProba \sup_{x,y} |\alpha_{\theta}(x,y) -
  \alpha_{\theta'}(x,y) | \\
  & +2 \InterractionProba S m^2 \, \left(\sup_{x,y} |g_{\theta}(x,y)- g_{\theta'}(x,y)| +
    \|\T - \T'\|_{\mathrm{TV}} + \sup_x |G_{\T'}(x) - G_\T(x)| \right) \eqsp.
\end{align*}
The term $|\alpha_{\theta}(x,y)-\alpha_{\theta'}(x,y)|$ is equal to $|1 \wedge
a - 1 \wedge b|$ with
\begin{align*}
  a= \frac{\pi^{ \beta_{k+1}-\beta_k}(y) \int g_{\theta} (x,z)
    \theta(\rmd z)}{\pi^{\beta_{k+1}-\beta_k}(x)\int g_{\theta} (y,z) \theta(\rmd z)}
  \quad \textrm{and} \quad b= \frac{\pi^{\beta_{k+1}-\beta_k}(y) \int
    g_{\theta'} (x,z) \theta'(\rmd z)}{\pi^{\beta_{k+1}-\beta_k}(x)\int
    g_{\theta'} (y,z) \theta'(\rmd z)} \eqsp.
\end{align*}
Note that $ |1 \wedge a - 1 \wedge b| \leq |b-a| \left( \un_{\{b \leq 1, a > 1
    \}} + \un_{a \leq 1} \right)$. Therefore, for all $\theta,\theta' \in \tetaset_m$,
\[
\sup_{x,y}| \alpha_\theta(x,y) - \alpha_{\theta'}(x,y)| \leq S^2 m^2 \
\left(\sup_{x,y}| g_\theta(x,y) - g_{\theta'}(x,y) | + \|\theta-\theta' \|_{\mathrm{TV}}
\right) \eqsp.
\]
The term $|G_{\T'}(x) - G_\T(x)|$ is upper bounded  by
\begin{align*}
 |G_{\T'}(x) - G_\T(x)| \leq \sup_{x,y} | g_\theta(x,y) - g_{\theta'}(x,y) | + S \|\theta-\theta' \|_{\mathrm{TV}} \eqsp.
\end{align*}
Moreover,
 \[
   \left| g_{\theta}(x,y)-g_{\theta'}(x,y) \right| = \left| \sum_{l=1}^S
     [h_{\theta,l}(x)h_{\theta,l}(y)-h_{\theta',l}(x)h_{\theta',l}(y)] \right|
   \leq 2S \sup_{l,x} |h_{\theta,l}(x)-h_{\theta',l}(x)| \eqsp.
\]
This concludes the proof. \end{proof}

\subsection{Proof of R\ref{hyp:proof:lln}($k+1$)}
Let $a \in (0,\frac{1+\vitborn}{2} \wedge 1)$ and set $V = W^a$. We check the
conditions of \cite[Theorem 2.7]{fort:moulines:priouret:2010}.  By
Proposition~\ref{drift:small:aees}, condition A3 of
\cite{fort:moulines:priouret:2010} holds.  By R\ref{hyp:proof:pi}($k+1$),
$\lim_{n \to \infty} \pi_{\theta_n}(f) = \pi_{\T_\star}(f)$
w.p.1 for any continuous function $f$ in $\mathcal{L}_{W^a}$. Condition A4
(resp.  A5) of \cite{fort:moulines:priouret:2010} is proved in
Lemma~\ref{lemma:A5} (resp.  Lemma~\ref{lemma:a6}).

\begin{lemma}
\label{lemma:A5}
Assume E\ref{EE1}, E\ref{EE3}, E\ref{EE:ring}, R\ref{hyp:proof:lln}(k),
R\ref{hyp:proof:subset}(j) and $\PE[W_j(Y_0^{(j)})]< \infty$ for all $j \leq
k$. Then for any $a \in (0,\frac{1+\vitborn}{2} \wedge 1)$
\[
\sum_{j \geq 1} j^{-1} (L_{\theta_j} \vee L_{\theta_{j-1}})^6
D_{W^a}(\theta_j,\theta_{j-1}) W^a(X_j) < \infty \ \qquad \PP-\text{a.s.,}
\]
where $L_{\theta} = C_{\theta} \vee (1-\rho_{\theta})^{-1}$.
\end{lemma}
\begin{proof}
  By R\ref{hyp:proof:subset}(j) for all $j \leq k$, it is sufficient to prove
  that for any positive integer $q$
\[
\sum_{j \geq 1} j^{-1} (L_{\theta_j} \vee L_{\theta_{j-1}})^6
D_V(\theta_j,\theta_{j-1}) V(X_j) \ \un_{A_{q,j}^{(k)}} < \infty \qquad
\PP-\text{a.s.}
\] 
where $A_{q,j}^{(k)}$ is defined in Appendix~\ref{app:lemme:unif:moments}.
Following the same lines as in the proof of Lemma~\ref{lemme:unif:moments}, it
can be proved that $\sum_{j=1}^{q} j^{-1} (L_{\theta_j} \vee
L_{\theta_{j-1}})^6 D_V(\theta_j,\theta_{j-1}) V(X_j) < \infty$ w.p.1.

By Lemma~\ref{lemme:unif:ergo} and R\ref{hyp:proof:lln}(k), on the set
$\bigcap_{l \geq q} \{\theta_l \in \tetaset_{m_\star} \}$, $\limsup_{n}
L_{\theta_n} < \infty$ w.p.1. Therefore, we have to prove that $ \sum_{j \geq q}
j^{-1} D_V(\theta_j,\theta_{j-1}) V(X_j) \un_{A_{q,j}^{(k)}} < \infty$ w.p.1.
Following the same lines as in the proof of Lemma~\ref{lemme:cv:dv:proba}, we
obtain that on the set $A_{q,j}^{(k)}$, there exists a constant $C_m$ such that
\begin{align*}
  D_V(\theta_j,\theta_{j-1}) & \leq C_m \left( \sup_{l} \left|
      \bornea_{\theta_j,l} -\bornea_{\theta_{j-1},l} \right| + \| \theta_j -
    \theta_{j-1} \|_{\mathrm{TV}} \right) \left(\|\theta_j\|_V
    +    \|\theta_{j-1}\|_V \right) \\
  & + C_m \ \| \theta_j-\theta_{j-1}\|_V \eqsp.
\end{align*}
Set $s,\gamma$, such that $s=1 \vee (2a)<1+\vit < 1 +\vitborn$. By E\ref{EE:ring}\ref{ring:b}, there
exists a r.v. $Z$ finite w.p.1 such that $\PP$-a.s.  
\begin{align*}
  |\bornea_{\T_n,l}-\bornea_{\T_{n-1},l}| + \| \theta_n - \theta_{n-1}
  \|_{\mathrm{TV}} \leq Z \left( \frac{1}{ n^{\vit}} +
    \frac{1}{n} \right) \eqsp.
\end{align*}
Therefore,  it holds
\begin{multline*}
  \mathcal{I}_\vit \eqdef \PE\left[\left( \sum_{j \geq q} j^{-1}
      \left( j^{-\vit} + j^{-1} \right) \left(\|\theta_j\|_V + \|\theta_{j-1}\|_V \right)V(X_j) \un_{A_{q,j}^{(k)}} \right)^{\frac{1}{s}} \right] \\
  \leq \sum_{j \geq q} j^{-1/s}  \left( j^{-\vit/s} + j^{-1/s} \right) \PE\left[ 
    \left(\|\theta_j\|_V + \|\theta_{j-1}\|_V \right)^{\frac{1}{s}}  V^{\frac{1}{s}}(X_j) \un_{A_{q,j}^{(k)}} \right] \eqsp.
\end{multline*}
We have,
\begin{equation*}
\mathcal{I}_\vit \leq 2  C(\vit) \ \sup_j\PE\left[\left(\|\theta_j\|_V \right)^{\frac{2}{s}}  \un_{A_{q,j}^{(k)}}
   \right]^{1/2} \ \sup_j \PE\left[V(X_j)^{\frac{2}{s}}
 \un_{A_{q,j}^{(k)}}
\right]^{1/2} \  \eqsp,
\end{equation*}
where $C( \vit) \eqdef \sum_{j \geq q} \left(
  j^{(-1-\vit)/s} + j^{-2/s} \right)$ is finite since $2/s > 1$ and
$1+ \vit > s$.  Since $V^{2/s} \leq W$, Lemma~\ref{lemme:unif:moments}
implies that $\sup_j \PE\left[ W(X_j) \un_{A_{q,j}^{(k)}}  \right] < \infty$.  In addition, since $2/s
> 1$ we have, by Jensen's inequality,
\begin{align*}
  \PE\left[\|\theta_j\|_V^{\frac{2}{s}} \un_{A_{q,j}^{(k)}} \right] & \leq
  \PE\left[\left( \frac{1}{j} \sum_{p=1}^j V(Y_p) \right)^{\frac{2}{s}}
    \un_{A_{q,j}^{(k-1)}} \right] \leq \PE\left[ \frac{1}{j} \sum_{p=1}^j
    V^{\frac{2}{s}}(Y_p) \un_{A_{q,j}^{(k-1)}} \right] \\
& \leq \sup_p \PE\left[
    W(Y_p) \un_{A_{q,p-1}^{(k-1)}} \right]
\end{align*}
which is finite under Lemma~\ref{lemme:unif:moments}. Similarly, we prove that $
\sum_{j \geq q} j^{-1} \| \theta_j-\theta_{j-1}\|_V V(X_j)
 \un_{A_{q,j}^{(k)}} < \infty$ w.p.1,
upon noting that $ \| \theta_j-\theta_{j-1}\|_V \leq j^{-1}(V(Y_{j}) +
\theta_{j-1}(V))$.
\end{proof}
\begin{lemma} \label{lemma:a6}
  Assume E\ref{EE1}, E\ref{EE3}, E\ref{EE:ring}\ref{ring:c}-\ref{ring:a},
  R\ref{hyp:proof:lln}(k), R\ref{hyp:proof:subset}(j) and $\PE[W_j(Y_0^{(j)})]<
  \infty$ for all $j \leq k$. For any $a \in (0,1)$,
\[
\sum_{j \geq 1} j^{-1/a} L_{\theta_j}^{2/a} \ P_{\theta_j} W(X_j) < \infty
\eqsp, \qquad \PP-\text{a.s.}
\]
\end{lemma}
\begin{proof}
 By R\ref{hyp:proof:subset}(j) for all $j \leq k$, it is sufficient to prove
  that for any positive integer $q$
\[
\sum_{j \geq 1} j^{-1/a} L_{\theta_j}^{2/a} \ P_{\theta_j} W(X_j) \ 
\un_{A_{q,j}^{(k)}} < \infty \qquad \PP-\text{a.s.}
\] 
where $A_{q,j}^{(k)}$ is defined in Appendix~\ref{app:lemme:unif:moments}.  Let
$q \geq 1$.  By Lemma~\ref{lemme:unif:ergo}, $\sup_j L_{\theta_{j}}
\un_{A_{q,j}^{(k)}} < \infty$ w.p.1; and, as in the proof of
Lemma~\ref{lemme:unif:moments}, it can be proved that $\sup_j \PE\left[
  P_{\theta_j} W(X_j) \un_{A_{q,j}^{(k)} } \right] < \infty$.  The proof is
concluded since  $\sum_k k^{-1/a} < \infty$.
\end{proof}


\subsection{Proof of Proposition~\ref{prop:vitesse:quantiles}}
\label{secproof:prop:vitesse:quantiles}
 The proof
uses a Hoeffding inequality for (non-stationary) Markov chains.  The following
result is proved in
\cite[section 5.2, theorem 17]{douc:moulines:olsson:vanhandel:2011}.
\begin{proposition} \label{hoeffding}
  Let $(Y_k)_{k \in \mathbb{N}}$ be a Markov chain on $(\X, \mathcal{X})$, with
  transition kernel $Q$ and initial distribution $\eta$.  Assume $Q$ is
  $\overline{W}$-uniformly ergodic, and denote by $\T_{\star}$ its unique invariant
  distribution. Then there exists a constant $K$ such that for any $t>0$ and for any
  bounded function $f:\X \rightarrow \mathbb{R}$
\begin{displaymath}
 \PP\left( \sum_{i=1}^n f(Y_i) - n \T_{\star}(f) \geq t \right) \leq K \eta(\overline{W}) \exp \left[-\frac{1}{K} \left( \frac{t^2}{n |f|_{\infty}^2} \wedge \frac{t}{|f|_{\infty}} \right) \right] \eqsp.
\end{displaymath}
\end{proposition}
\begin{lemma} \label{lemme:appl:hoeffding}
  Assume that there exists $\overline{W}$ such that $\{Y_n, n \geq 0\}$ is a
  $\overline{W}$-uniformly ergodic Markov chain with initial distribution $\eta$
  with $\eta(\overline{W}) < \infty$.  Let $l \in \{1, \cdots, S-1 \}$ and $p_l \in
  (0,1)$; and set $\xi_l = F_{\theta_\star}^{-1}(p_l)$.  For all $\epsilon >0$
  and any $ n\geq 1$,
\begin{align*}
 \PP \left(|\bornea_{\T_n,l}-\xi_l|>\epsilon\right) \leq 2 K \eta(\overline{W}) \exp \left(- \frac{n}{K} \left(\delta_{\epsilon}^2 \wedge \delta_{\epsilon} \right) \right) \eqsp,
\end{align*}
where $\delta_{\epsilon} = min \left\{ F_{\T_{\star}}(\xi_l +\epsilon) -
  p_l,p_l-F_{\T_{\star}}(\xi_l-\epsilon) \right\}$.
\end{lemma}
\begin{proof}
  Let $\epsilon>0$. We write $  \PP \left(|\bornea_{\T_n,l}-\xi_l|>\epsilon \right)  \leq \PP
  \left(\bornea_{\T_n,l} \geq \xi_l + \epsilon \right) + \PP
  \left(\bornea_{\T_n,l} < \xi_l - \epsilon \right)$.
Since $F_{\T_{n}}(x) \leq t$ iff $x\leq F_{\T_{n}}^{-1}(t)$,
\begin{align*}
  \PP \left(\bornea_{\T_n,l} \geq \xi_l + \epsilon \right) &= \PP
  \left(F_{\theta_n}^{-1}(p_l) \geq \xi_l + \epsilon \right) = \PP \left( p_l \geq F_{\T_{n}}(\xi_l + \epsilon) \right)\\
 &
= \PP \left(\sum_{k=1}^n \un_{\{\mesobs(Y_k) > \xi_l + \epsilon \}}  \geq n(1-p_l)\right) \eqsp.
\end{align*}
Proposition~\ref{hoeffding} is then applied with $f(x) = \un_{\{\mesobs(x)>\xi_l +
  \epsilon\}}$. As
\begin{align*}
 \T_{\star}(f) = \int \un_{\{\mesobs(x)>\xi_l + \epsilon\}} \T_{\star}(dx) = 1 - F_{\T_{\star}}(\xi_l + \epsilon) \eqsp,
\end{align*}
we obtain
\begin{align*}
  \PP \left(\bornea_{\T_n,l} \geq \xi_l + \epsilon \right)
  &=\PP \left( \sum_{k=1}^n f(Y_k) - n \T_{\star}(f) \geq n \left( F_{\T_{\star}}(\xi_l + \epsilon) - p_l\right) \right)\\
  & \leq K \eta(\overline{W}) \exp \left( - \frac{n}{K} \left[ \left(
        F_{\T_{\star}}(\xi_l + \epsilon) - p_l \right)^2 \wedge \left(
        F_{\T_{\star}}(\xi_l + \epsilon) -p_l \right) \right]\right) \eqsp.
\end{align*}
for some constant $K$ independent of $n,l,\epsilon$.  Similarly,
\begin{align*}
  \PP \left(\bornea_{\T_n,l} < \xi_l - \epsilon \right) \leq K \eta(\overline{W}) \, \exp \left( - \frac{n}{K} \left[ \left( p_l
        -F_{\T_{\star}}(\xi_l - \epsilon) \right)^2 \wedge \left( p_l
        -F_{\T_{\star}}(\xi_l - \epsilon) \right)\right] \right) \eqsp,
\end{align*}
which concludes the proof.
\end{proof}

{\em Proof of Proposition~\ref{prop:vitesse:quantiles}} Let $f_{\T_{\star}} =
F'_{\T_{\star}}$ and $\epsilon_n$ be defined by
\begin{align*}
  \epsilon_n = \frac{2 \sqrt{2}}{f_{\T_{\star}}(\xi_l)} \sqrt{K} \sqrt{\frac{\log(n)}{n}} \eqsp,
\end{align*}
where $K$ is given by Lemma~\ref{lemme:appl:hoeffding}.  Note that under
(\ref{hyp:fd}), $f_{\T_{\star}}(\xi_l)>0$ since $p_l \in (0,1)$. By
(\ref{hyp:fd}), $F_{\T_{\star}}$ is differentiable and we write when
$n \to \infty$
 \begin{align*}
   F_{\T_{\star}}(\xi_l + \epsilon_n) - p_l = F_{\T_{\star}}(\xi_l +
   \epsilon_n) - F_{\T_{\star}}(\xi_l) = f_{\T_{\star}}(\xi_l) \epsilon_n +
   o(\epsilon_n) \eqsp.
\end{align*}
Hence $F_{\T_{\star}}(\xi_l + \epsilon_n) - p_l \geq \sqrt{2 K}
\sqrt{\frac{\log(n)}{n}}$ for $n$ large enough. Similarly,
$p_l-F_{\T_{\star}}(\xi_l - \epsilon_n)\geq \sqrt{2 K}
\sqrt{\frac{\log(n)}{n}}$ for $n$ large enough. So when $n$ is large enough, $
n K^{-1} \left( \delta_{\epsilon_n}^2 \wedge \delta_{\epsilon_n} \right) \geq 2
\log(n)$ with $\delta_\epsilon$ defined in Lemma~\ref{lemme:appl:hoeffding}.
By Lemma~\ref{lemme:appl:hoeffding}, for $n$ large enough, to
\begin{align*}
  \PP \left(|\bornea_{\T_n,l}-\xi_l|>\epsilon_n \right) \leq \frac{2 K \eta(\bar W)}{n^2} \eqsp.
\end{align*}
As $\sum_{n=1}^{\infty} \PP \left(|\bornea_{\T_n,l}-\xi_l|>\epsilon_n \right) <
\infty$, the Borel-Cantelli lemma yields $ \limsup_n \epsilon_n^{-1}
\left|\bornea_{\T_n,l} - \xi_l \right| < \infty$ w.p.1. This concludes the
proof.


\end{document}